\documentclass[reqno,11pt,a4paper]{article}

\usepackage{amsmath,amssymb,amsthm}
\usepackage[UKenglish]{babel}
\usepackage[right=2.54cm,left=2.54cm]{geometry}
\usepackage{ifthen}
\usepackage[utf8]{inputenc}
\usepackage{sectsty}
\usepackage[pdftitle={On stopping Fock-space processes},
            pdfauthor={Alexander C. R. Belton},
            bookmarks=true,
            colorlinks=false,
            pdfstartview=FitH,
            unicode]{hyperref}

\sectionfont{\large}
\subsectionfont{\normalsize}

\mathchardef\ordinarycolon\mathcode`\:
\mathcode`\:=\string"8000
\def\vcentcolon{\mathrel{\mathop\ordinarycolon}}
\begingroup \catcode`\:=\active
\lowercase{\endgroup
\let :\vcentcolon}

\DeclareFontFamily{U}{mathx}{\hyphenchar\font45}
\DeclareFontShape{U}{mathx}{m}{n}{<5> <6> <7> <8> <9> <10> %
<10.95> <12> <14.4> <17.28> <20.74> <24.88> mathx10}{}
\DeclareSymbolFont{mathx}{U}{mathx}{m}{n}
\DeclareFontSubstitution{U}{mathx}{m}{n}
\DeclareMathAccent{\widecheck}{0}{mathx}{"71}

\theoremstyle{plain}
\newtheorem{theorem}{Theorem}[section]
\newtheorem{lemma}[theorem]{Lemma}
\newtheorem{proposition}[theorem]{Proposition}
\newtheorem{corollary}[theorem]{Corollary}

\theoremstyle{definition}
\newtheorem{definition}[theorem]{Definition}
\newtheorem{example}[theorem]{Example}
\newtheorem{notation}[theorem]{Notation}
\newtheorem{remark}[theorem]{Remark}
\newtheorem{question}[theorem]{Question}

\newenvironment{mylist}%
{\begin{list}{}%
{\leftmargin 4.75em\labelwidth 3.5em\rightmargin 0.8em%
\topsep 0.75ex\itemsep 0.25ex}}%
{\end{list}}
{\begin{list}{}{}}{\end{list}}

\let\origthebibliography=\thebibliography
\def\thebibliography{\renewcommand{\section}[2]{}\origthebibliography}

\newcommand{\bop}[2]%
{\ifthenelse{\equal{#2}{}}{\bopp( #1 )}{\bopp( #1; #2 )}}
\newcommand{\bopp}{B}
\newcommand{\borel}{\mathcal{B}}
\newcommand{\bra}[1]{\langle#1|}
\newcommand{\elltwo}{L^2( \R_+; \mul )}
\newcommand{\esssup}{\mathop{\mathrm{ess\ sup}}}
\newcommand{\evec}[1]{\evecc(#1)}
\newcommand{\evecc}{\varepsilon}
\newcommand{\evecs}{\mathcal{E}}
\newcommand{\expn}{\mathbb{E}}
\newcommand{\fock}{\mathcal{F}}
\newcommand{\hilb}{\mathsf{H}}
\newcommand{\hlf}{\mbox{$\frac12$}}
\newcommand{\im}{\mathop{\mathrm{im}}}
\newcommand{\ket}[1]{|#1\rangle}
\newcommand{\mul}{\mathsf{k}}
\renewcommand{\Pr}{\mathbb{P}}
\newcommand{\probsp}{( \samplesp, \salg, \Pr )}
\newcommand{\qsto}{\Rightarrow}
\newcommand{\rd}{\mathrm{d}}
\newcommand{\salg}{\mathcal{A}}
\newcommand{\samplesp}{\Omega}
\newcommand{\sexp}[1]{\ssexp(#1)}
\newcommand{\ssexp}{\zeta}
\newcommand{\std}{\,\rd}
\newcommand{\stlim}{\mathop{\mathrm{st.lim}}\limits}
\newcommand{\Vac}{\Omega}
\newcommand{\wc}[1]{\widecheck{#1}}
\newcommand{\wh}[1]{\widehat{#1}}

\renewcommand{\C}{\mathbb{C}}
\newcommand{\N}{\mathbb{N}}
\newcommand{\R}{\mathbb{R}}
\newcommand{\Z}{\mathbb{Z}}

\renewcommand{\ge}{\geqslant}
\renewcommand{\le}{\leqslant}

\newcommand{\Cf}{\textit{Cf.\ }}
\newcommand{\cf}{\textit{cf.\ }}
\newcommand{\eg}{\textit{e.g., }}
\newcommand{\ie}{\textit{i.e., }}

\newcommand{\tu}[1]{\textup{#1}}

\numberwithin{equation}{section}

\begin{document}

\begin{center}
{\LARGE On stopping Fock-space processes}

\vspace*{1ex}
{\large Alexander C.~R.~Belton}\\[0.5ex]
{\small Department of Mathematics and Statistics\\
Lancaster University, United Kingdom\\[0.5ex]
\textsf{a.belton@lancaster.ac.uk}}\qquad%
{\small Tuesday 31st July 2018}
\end{center}

\begin{abstract}\noindent
We consider the theory of stopping bounded processes within the
framework of Hudson--Parthasarathy quantum stochastic calculus, for
both identity and vacuum adaptedness. This provides significant new
insight into Coquio's method of stopping
(J.~Funct.\ Anal.~238:149--180, 2006). Vacuum adaptedness is required
to express certain quantum stochastic representations, and many
results, including the proof of the optional-sampling theorem, take a
more natural form.
\end{abstract}



\section{Introduction}

The extension of the notion of stopping time from classical to
non-commutative probability is straightforward, with the earliest
definition in the literature due to Hudson \cite{Hud79}. The idea was
developed in the setting of the Clifford probability gauge space by
Barnett and Lyons \cite{BaL86}, and for abstract filtered von~Neumann
algebras by Barnett and Thakrar \cite{BaT87,BaT90}, and
Barnett and Wilde \cite{BaW90,BaW93}; see also \cite{Sav88}, where
Sauvageot initiated a programme to solve a $C^*$-algebraic version of
the Dirichlet problem, and recent work by {\L}uczak \cite{Luc05}.

In the more concrete setting of Hudson--Parthasarathy quantum
stochastic calculus~\cite{HuP84}, an extensive theory was developed by
Parthasarathy and Sinha \cite{PaS87}. In particular, they showed that,
given a quantum stopping time~$S$, the Boson Fock space
$\fock = \fock_+\bigl( \elltwo \bigr)$ has the factorisation
$\fock_{S)} \otimes \fock_{[S}$, where the spaces $\fock_{S)}$ and
$\fock_{[S}$ are pre-$S$ and post-$S$ spaces; this provides a form of
the strong Markov property that generalises Hudson's result
\cite{Hud79}. Further contributions in this setting have been made by
Meyer \cite{Mey87}, Accardi and Sinha \cite{AcS89}, Attal and Sinha
\cite{AtS98}, Sinha~\cite{Sin03}, Attal and Coquio \cite{AtC04},
Coquio \cite{Coq06} and Hudson \cite{Hud07,Hud10}. Quantum stopping
times are applied to the CCR flow and its cocycles in \cite{BeS14} and
\cite{BeS16}, building on work of Applebaum \cite{App88}.
Parthasarathy and his collaborators developed the theory around
another quantum-probabilistic version of the Dirichlet problem:
see~\cite{AtP95}, \cite{BhP95}, \cite{AtP96} and~\cite{Par03}.

Here, we investigate stopping in parallel for both identity-adapted
and vacuum-adapted versions of quantum stochastic calculus; the latter
variant was introduced in~\cite{Blt01} and further developed
in~\cite{Blt04}. By developing the two forms of the theory in
parallel, we provide insights which illuminate Coquio's constructions
in~\cite{Coq06} and allow us to extend her results. Certain integrals,
which seem somewhat mysterious from the identity-adapted viewpoint,
can be seen to be products due to switching forms of
adaptedness. Furthermore, the vacuum-adapted theory is seen to have
superior properties, which correspond more naturally to the classical
situation. For simplicity and clarity of exposition, we restrict our
attention to processes composed of bounded operators.

In the Fock-space context, a quantum stopping time $S$ is a
projection-valued measure on the extended half~line $[ 0, \infty ]$,
such that $S\bigl( [ 0, \infty ] \bigr) = I$, the identity operator on
$\fock$, and $t \mapsto S\bigl( [ 0, t ] \bigr)$ is an
identity-adapted process, \ie
$S\bigl( [ 0, t ] \bigr) \in \bop{\fock_{t)}}{} \otimes I_{[t}$ for
all $t \ge 0$. (Here the Boson Fock space~$\fock$ is identified with
$\fock_{t)} \otimes \fock_{(t}$; this is the familiar deterministic
version of the Parthasarathy--Sinha factorisation.)
Section~\ref{sec:qst} presents this definition, derives some basic
results from it and sets out some classical and non-commutative
examples.

The primary object associated with a quantum stopping time $S$ is its
time projection $E_S$, and this is introduced in
Section~\ref{sec:timeproj}; if~$S$ is deterministic, so
that~$S\bigl( \{ t \} ) = I$ for some $t$, then $E_S$ is the
orthogonal projection $E_t = I_{t)} \otimes P^\Vac_{[t}$ onto
$\fock_{t} \otimes \evec{0}$, where here $\evec{0}$ is the vacuum
vector in $\fock_{[t}$. If $\fock$ is identified with the $L^2$~space
of a standard Brownian motion with filtration
$\salg = ( \salg_t )_{t \in \R_+}$ and~$S$ corresponds to a classical
$\salg$-stopping time $\tau$ then $E_S$ becomes the conditional
expectation with respect to the stopping time $\sigma$-algebra
$\salg_\tau$. The time projection $E_S$  has a quantum stochastic
representation (Theorem~\ref{thm:qstint}):
\begin{equation}\label{eqn:integral_rep}
E_S = I + \int_0^\infty I_\mul \otimes %
S\bigl( ( s, \infty ] \bigr) E_s \std \Lambda_s.
\end{equation}
The integrand is a vacuum-adapted process, which shows that this form
of adaptedness appears naturally when considering quantum stopping
times.

If $S$ is a quantum stopping time and $X = ( X_t )_{t \in \R_+}$ is a
process (\ie a family of operators on~$\fock$ satisfying suitable
measurability and adaptedness conditions) then there are three natural
approaches to stopping $X$ at $S$: these are
\begin{align*}
S \cdot X & := S\bigl( \{ 0 \} \bigr) X_0 + %
\lim_\pi \sum_{j = 0}^\infty S\bigl( ( \pi_{j - 1}, \pi_j ] \bigr) %
X_{\pi_j}, \tag*{(left)} \\[1ex]
X \cdot S & := X_0 S\bigl( \{ 0 \} \bigr) + %
\lim_\pi \sum_{j = 0}^\infty X_{\pi_j} %
S\bigl( ( \pi_{j - 1}, \pi_j ] \bigr) \tag*{(right)} \\[1ex]
\text{and} \qquad S \cdot X \cdot S & := %
S\bigl( \{ 0 \} \bigr) X_0 S\bigl( \{ 0 \} \bigr) + \lim_\pi %
\sum_{j = 0}^\infty S\bigl( ( \pi_{j - 1}, \pi_j ] \bigr)
X_{\pi_j} S\bigl( ( \pi_{j - 1}, \pi_j ] \bigr), \tag*{(double)}
\end{align*}
assuming these limits, taken over partitions of $\R_+$, exist in an
appropriate sense. (To establish convergence is often difficult, or
apparently impossible in general.) The time projection $E_S$ is the
result of stopping the vacuum-adapted process $( E_t )_{t \in \R_+}$
in any of these three senses. Each of the expressions yields the same
orthogonal projection, and convergence holds in the strong operator
topology, since these projections form an decreasing family as the
partition is refined: see Theorem~\ref{thm:expS} below.

In the vacuum-adapted setting, the value at time $t$ of a martingale
$M$ closed by the operator~$M_\infty$ is simply $E_t M_\infty E_t$;
see Section~\ref{sec:processes} for the definitions of martingales in
this context. Thus it is natural to define $M_{\wc{S}}$, the value of
the martingale $M$ stopped in a vacuum-adapted manner at $S$, to be
$E_S M_\infty E_S$; it is easy to see that this equals $S \cdot M
\cdot S$, the result of double stopping $M$ at $S$. The issue of
convergence becomes that of the existence of $E_S$, which is long
established, and this definition has various good properties: see
Sections~\ref{sec:algebras} and~\ref{sec:closedmgs} below. In
particular, the optional-sampling theorem,
Theorem~\ref{thm:vacoptional}, holds, and is a immediate consequence
of the identity $E_S \wedge E_T = E_{S \wedge T}$, which is true for
any two quantum stopping times~$S$ and $T$
(Theorem~\ref{thm:exporder}).

In \cite{Coq06}, Coquio puts forward a method of stopping for
identity-adapted processes which is not obviously one of the forms
given above. She begins by working with discrete stopping times, \ie
those with finite support: if $T$ is a quantum stopping time with
support $\{ t_1 < \cdots < t_n \} \subseteq \R_+$ and $X$ is a process
then
\begin{equation}\label{eqn:iddiscstop}
X_{\wh{T}} := \sum_{i, j = 1}^n \wh{\pi}( E_{t_i} )_{t_i \vee t_j} %
T\bigl( \{ t_i \} \bigr) %
\wh{\pi}( X_{t_i \vee t_j} )_{t_i \vee t_j} T\bigl( \{ t_j \} \bigr) %
\wh{\pi}( E_{t_j} )_{t_i \vee t_j}
\end{equation}
is the result of applying identity-adapted stopping to $X$ at $T$,
where $\wh{\pi}$ is the projection onto the space of identity-adapted
processes. For example, the operator
$\wh{\pi}( E_s )_t = I_{s)} \otimes P^\Vac_{[ s, t )} \otimes I_{[t}$
maps $\fock_{[ s, t )}$ onto the vacuum subspace and acts as
the identity on $\fock_{s)}$ and $\fock_{[t}$, with $\fock$ identified
with $\fock_{s)} \otimes \fock_{[ s, t )} \otimes \fock_{[t}$. (This
notation is explained further in Section~\ref{sec:processes}.) In
Section~\ref{sec:discrete} the vacuum-adapted version of this
definition is introduced and various consequences are derived; in
particular, Lemma~\ref{lem:discmg} shows that, for a closed
martingale, it agrees with the natural definition of $M_{\wc{T}}$
described above. In particular, Coquio's method is seen to be a form
of double stopping, at least for a large class of processes.

Working from the definition (\ref{eqn:iddiscstop}), Coquio obtains an
integral formula for the stopped process $M_{\wh{T}}$ when~$M$ is a
closed martingale (see Theorem~\ref{thm:idclosedmg}), and uses this to
extend the definition of $M_{\wh{T}}$ to arbitrary stopping times. The
key step in this result, Theorem~\ref{thm:idstopop} below, is to show
that, given any quantum stopping time $S$ and any $Z \in
\bop{\fock}{}$, the sum
\[
E_S Z E_S + \int_0^\infty I_\mul \otimes %
S\bigl( [ 0, s ] \bigr) \wh{\pi}\bigl( E_S Z E_S \bigr)_s %
S\bigl( [ 0, s ] \bigr) \std \Lambda_s
\]
extends to a bounded operator $Z_{\wh{S}}$; we provide a somewhat
shorter version of Coquio's proof in Section~\ref{sec:closedmgs}. It
follows that the difference between stopping a closed martingale $M$
at an arbitrary quantum stopping time $S$ in the identity-adapted and
vacuum-adapted senses is given by a gauge integral:
\[
M_{\wh{S}} - M_{\wc{S}} = \int_0^\infty I_\mul \otimes 
S\bigl( [ 0, s ] \bigr) \wh{\pi}\bigl( M_{\wc{S}} \bigr)_s %
S\bigl( [ 0, s ] \bigr) \std \Lambda_s;
\]
in particular, the integral in Coquio's definition of $M_{\wh{S}}$ can
be seen as an artifact produced by working with identity, rather than
vacuum, adaptedness.

In Section~\ref{sec:FV}, vacuum-adapted stopping for FV processes and
for semimartingales is extended from discrete to arbitrary times. An
FV process $Y$ is of the form $Y_t = \int_0^t H_s \std s$ for some
integrand process $H$, and a semimartingale is the sum of a martingale
and an FV process.

Given sufficient regularity, a semimartingale may be written as the
sum of four quantum stochastic integrals. Such a process is called a
regular quantum semimartingale if identity adapted, and a regular
$\Omega$-semimartingale if vacuum adapted. The integral formula
(\ref{eqn:integral_rep}) for~$E_S$ is used in Section~\ref{sec:vqsm}
to show that the class of regular $\Omega$-semimartingales is closed
under vacuum-adapted stopping (Theorem~\ref{thm:vacsmg}); Coquio has
obtained the analogous result for regular quantum semimartingales
\cite[Proposition~3.16]{Coq06}.

In the final Section~\ref{sec:postS}, it is shown that the projection
onto the future space may be stopping in the three different senses,
and that a natural relationship holds between these stopped
operators. Each has a vacuum-adapted quantum stochastic
representation.

An appendix, Section~\ref{sec:qsc}, is included to gather the
necessary results on quantum stochastic integration.

\subsection{Acknowledgement}

The author is grateful for the referee's comments on a previous
version of this work.

\subsection{Notation and conventions}

The term ``increasing'' applies in the weak sense. Hilbert spaces have
complex scalar fields; inner products are linear in the second
argument. The indicator function of the set $A$ is denoted by
$1_A$. The complement of an orthogonal projection $P$ is denoted by
$P^\perp$. The set of natural numbers $\N = \{ 1, 2, 3, \ldots \}$,
the set of non-negative integers $\Z_+ = \N \cup \{ 0 \}$ and the set
of non-negative real numbers $\R_+ = [ 0, \infty )$. The von~Neumann
algebra of bounded operators on a Hilbert space $\hilb$ is denoted by
$\bop{\hilb}{}$. Given a real number $x$, its ceiling
$\lceil x \rceil$ is the smallest integer at least as great as~$x$.

\section{Quantum stopping times}\label{sec:qst}

\begin{notation}
Let $\fock_A = \fock_+\bigl( L^2( A; \mul ) \bigr)$ denote Boson
Fock space over $L^2( A; \mul )$, where $A$ is a subinterval of~$\R_+$
and $\mul$ is a complex Hilbert space. For brevity, let
$\fock := \fock_{\R_+}$, $\fock_{t)} := \fock_{[ 0, t )}$ for all
$t \in ( 0, \infty )$ and~$\fock_{[t} := \fock_{[ t, \infty )}$ for
all $t \in \R_+$, with similar abbreviations $I$, $I_{t)}$
and~$I_{[t}$ for the identity operators on these spaces.

The set of exponential vectors $\{ \evec{f} : f \in \elltwo \}$ is
total in $\fock$ and linearly independent, with
$\langle \evec{f}, \evec{g} \rangle = \exp \langle f, g \rangle$ for
all $f$, $g \in \elltwo$; we let~$\evecs$ denote the linear span of this
set.
\end{notation}

\begin{definition}[{\cite[Section~3]{PaS87}}]\label{def:qst}
A \emph{spectral measure} on $\borel[ 0, \infty ]$, the Borel
$\sigma$-algebra on the extended half-line
$[ 0, \infty ] := \R_+ \cup \{ \infty \}$, is a map
\[
S : \borel[ 0, \infty ] \to \bop{\hilb}{},
\]
where $\hilb$ is a complex Hilbert space, such that
\begin{mylist}
\item[(i)] the operator $S( A )$ is an orthogonal projection for all
$A \in \borel[ 0, \infty ]$;
\item[(ii)] the mapping
$\borel[ 0, \infty ] \to \C; \ %
A \mapsto \langle x, S( A ) y \rangle$
is a complex measure for all~$x$,~$y \in \fock$;
\item[(iii)] the total measure $S\bigl( [ 0, \infty ] \bigr) = I$.
\end{mylist}
The spectral measure $S$ is a \emph{quantum stopping time} if
$\hilb = \fock$ and $S$ is identity adapted in the following sense:
\begin{mylist}
\item[(iv)] the operator $S\bigl( \{ 0 \} \bigr) \in \C I$ and
$S\bigl( [ 0, t ] \bigr) \in \bop{\fock_{t)}}{} \otimes I_{[t}$ for
all $t \in ( 0, \infty )$.
\end{mylist}
\end{definition}

\begin{remark}[{\Cf\cite[Definitions~3.1]{BaL86}}]
A spectral measure on $\borel[ 0, \infty ]$ may also be defined to be
an increasing family of orthogonal projections
$( S_t )_{t \in [ 0, \infty ]}$ in~$\bop{\hilb}{}$ such that
$S_\infty = I$. The equivalence of these definitions is ensured by the
spectral theorem for unbounded self-adjoint operators.
\end{remark}

\begin{proposition}\label{prp:qst}
Let $S$ be a spectral measure on $\borel[ 0, \infty ]$, and let
$A$ and $B$ be arbitrary elements of $\borel[ 0, \infty ]$.
\begin{mylist}
\item[\tu{(i)}] The operator $S( \emptyset ) = 0$.

\item[\tu{(ii)}] If $A \subseteq B$ then
$S( A ) S( B ) = S( A ) = S( B ) S( A )$.

\item[\tu{(iii)}] If $A$ and $B$ are disjoint then
$S( A \cup B ) = S( A ) + S( B )$ and $S( A ) S( B ) = 0$.

\item[\tu{(iv)}] In general,
$S( A ) S( B ) = S( A \cap B ) = S( B ) S( A )$.
\end{mylist}
\end{proposition}
\begin{proof}
This is a simple exercise.
\end{proof}

\begin{lemma}\label{lem:stmeas}
If $S : \borel[ 0, \infty ] \to \bop{\hilb}{}$ is a spectral measure
then $t \mapsto S\bigl( [ 0, t ] \bigr) x$ is right
continuous on $\R_+$, so strongly measurable, for all $x \in \hilb$.
\end{lemma}
\begin{proof}
Let $x \in \hilb$ and note that if $s$, $t \in \R_+$ with $s < t$ then
\[
\| S\bigl( [ 0, t ] ) x - S\bigl( [ 0, s ] \bigr) x \|^2 = %
\| S\bigl( ( s, t ] \bigr) x \|^2 = %
\langle x, S\bigl( ( s, t ] \bigr) x \rangle.
\]
Thus if $t_n \to s+$ then, since $\cap_n ( s, t_n ] = \emptyset$, so
$S\bigl( [ 0, t_n ] \bigr) x \to S\bigl( [ 0, s ] \bigr) x$. Hence
setting
\[
f_n : \R_+ \to \hilb; \ t \mapsto 1_{[ 0, n ]}( t ) %
S\bigl( [ 0, 2^{-n} \lceil 2^n t \rceil ] \bigr) x %
\qquad \text{for all } n \in \N
\]
gives a sequence of simple functions converging
pointwise to $t \mapsto S\bigl( [ 0, t ] \bigr) x$.
\end{proof}

The following result will be applied without comment.

\begin{corollary}\label{cor:meas}
Let $S : \borel[ 0, \infty ] \to \bop{\hilb}{}$ be a spectral measure.
\begin{mylist}
\item[\tu{(i)}] If $f : \R_+ \to \hilb$ is strongly measurable then so
is $t \mapsto S\bigl( [ 0, t ] \bigr) f( t )$.

\item[\tu{(ii)}] If $F : \R_+ \to \bop{\hilb}{}$ is pointwise strongly
measurable, so that $t \mapsto F( t ) x$ is strongly measurable for
all $x \in \hilb$, then so is
$t \mapsto F( t ) S\bigl( [ 0, t ] \bigr)$.
\end{mylist}
\end{corollary}
\begin{proof}
For all $n \ge 1$, let
\[
S_n( t ) := 1_{[ 0, n ]}( t ) %
S\bigl( [ 0, 2^{-n} \lceil 2^n t \rceil ] \bigr) = \left\{ %
\begin{array}{ll}
S\bigl( \{ 0 \} \bigr) & \mbox{if } t = 0, \\[1ex]
S\bigl( [ 0, j 2^{-n} ] \bigr) & \mbox{if } t \in %
\bigl( ( j - 1 ) 2^{-n}, j 2^{-n} ] \ %
( j = 1, \ldots, n 2^n ), \\[1ex]
0 & \mbox{if } t \in ( n, \infty ).
\end{array}\right.
\]
(i) If $N \subseteq \R_+$ is a Lebesgue-null set and
$( f_n )_{n \ge 1}$ is a sequence of simple functions such that
$\| f_n( t ) - f( t ) \| \to 0$ for all $t \in \R_+ \setminus N$ then
$t \mapsto S_n( t ) f_n( t )$ is a simple function and
\[
\| S_n( t ) f_n( t ) - S\bigl( [ 0, t ] \bigr) f( t ) \| \le %
\| f_n( t ) - f( t ) \| + %
\| ( S_n( t ) - S\bigl( [ 0, t ] \bigr) ) f( t ) \| \to 0
\]
if $t \in \R_+ \setminus N$.

(ii) Let $x \in \hilb$ and note that $t \mapsto F( t ) S_n( t ) x$ is
the sum of finitely many strongly measurable functions, so is
thus itself strongly measurable.
Since~$F( t ) S_n( t ) x \to F( t ) S\bigl( [ 0, t ] \bigr) x$ for
all~$t \in \R_+$, the claim follows.
\end{proof}

\begin{example}
For any $t \in [ 0, \infty ]$, setting
\[
t : \borel[ 0, \infty ] \to \bop{\fock}{}; \ %
A \mapsto \left\{ \begin{array}{cl}
 0 & \text{if } t \not\in A, \\[1ex]
 I & \text{if } t \in A
\end{array}\right.
\]
defines a quantum stopping time which corresponds to the deterministic
time~$t$.
\end{example}

\begin{example}\label{eg:classicalst}
Let $B = ( B_t )_{t \in \R_+}$ be a standard Brownian motion and use
the Wiener--It\^o--Segal transform to identify 
the Fock space $\fock$, where $\mul = \C$, with the space
$L^2\probsp$. If $\tau$ is a classical stopping time for $B$ then
\[
S : \borel[ 0, \infty ] \to \bop{\fock}{}; \ %
A \mapsto 1_{\{ \tau \in A \}}
\]
is a quantum stopping time, where the function $1_{\{ \tau \in A \}}$
acts by multiplication on $L^2\probsp$. The same applies with $B$
replaced by any classical process with the chaotic representation
property, \eg the classical or monotone Poisson processes, or
Az\'ema's martingale: see \cite[Section~II.1]{Att98}.
\end{example}

\begin{example}[{\cite[pp.508--510]{AtC04}}]\label{eg:poissonst}
Let $\nu = ( \nu_t )_{t \in \R_+}$ be a standard Poisson process with
intensity~$1$ and unit jumps on the probability space $\probsp$, where
$\salg$ is complete and generated by $\nu$; for all $n \in \Z_+$, let
\[
\tau_n := \inf\{ t \in \R_+ : \nu_t = n \}
\]
be the $n$th jump time. The fact that
\[
\int_0^t \phi_s( \omega ) \std \nu_s( \omega ) = %
\sum_{k = 1}^{\nu_t( \omega )} \phi_{\tau_k( \omega )}( \omega ) %
\qquad \text{for all } \omega \in \samplesp,
\]
where $\phi$ is any process, together with the identity
$\{ \nu_t \ge n \} = \{ \tau_n \le t \}$, implies that
\[
\int_0^t \bigl( 1_{\{ \tau_{n - 1} < s \}}( \omega ) - %
1_{\{ \tau_n < s \}}( \omega ) \bigr) \std \nu_s( \omega ) = %
\int_0^t 1_{( \tau_{n - 1}( \omega ), \tau_n( \omega ) ]}( s ) %
\std \nu_s( \omega ) = 1_{\{ \tau_n \le t \}}( \omega )
\]
for all $\omega \in \samplesp$, $n \in \N$ and $t \in \R_+$. Since
$( \nu_t - t )_{t \in \R_+}$ is a normal martingale and $\tau_m$ has
the gamma distribution with mean and variance $m$, it holds that
\[
\expn\Bigl[ \Bigl( \int_0^t 1_{\{ \tau_m = s \}} \std( \nu_s - s ) %
\Bigr)^2 \Bigr] = %
\expn\Bigl[ \int_0^t 1_{\{ \tau_m = s \}} \std s \Bigr] = %
\int_0^t \Pr( \tau_m = s ) \std s = 0
\]
and therefore, as elements of $L^2\probsp$,
\begin{equation}
1_{\{ \tau_n \le t \}} = %
\int_0^t \bigl( 1_{\{ \tau_{n - 1} \le s \}} - %
1_{\{ \tau_n \le s \}} \bigr) \std \nu_s \qquad \text{for all } %
t \in \R_+.
\end{equation}
With $\mul = \C$, let $N_t = \Lambda_t + A_t + A^\dagger_t + t I$ for
all $t \in \R_+$, so that $N$ is the usual quantum stochastic
representation of the Poisson process $\nu$: there exists an isometric
isomorphism $U_P : L^2\probsp \to \fock$ such that $U_P^* N_t U_P$ is
essentially self adjoint, with closure corresponding to multiplication
by $\nu_t$, for all $t \in \R_+$, and~$U_P 1 = \evec{0}$
\cite[Theorems~6.1--2]{HuP84}. It follows from Lemma~\ref{lem:poisint}
that
\begin{equation}\label{eqn:poissonst}
\int_0^t T_n\bigl( [ 0, s ] \bigr) \std N_s = %
U_P \int_0^t 1_{\{ \tau_n \le s \}} \std \nu_s \, U_P^* %
\qquad \text{on } \evecs
\end{equation}
for all $t \in \R_+$, where the quantum stopping time $T_n$ is defined
by setting
\[
T_n( A ) := U_P 1_{\{ \tau_n \in A \}} U_P^* \qquad \text{for all } %
n \in \Z_+ \text{ and } A \in \borel[ 0, \infty ],
\]
as in Example~\ref{eg:classicalst}, and the stochastic integral on the
right-hand side of (\ref{eqn:poissonst}) acts by multiplication
on~$L^2\probsp$. In particular, $T_n$ satisfies the quantum stochastic
differential equation
\begin{equation}
T_n\bigl( [ 0, t ] \bigr) = \int_0^t \Bigl( %
T_{n - 1}\bigl( [ 0, s ] \bigr) - T_n\bigl( [ 0, s ] \bigr) %
\Bigr) \std N_s %
\qquad \text{for all } t \in \R_+ \text{ and } n \in \N.
\end{equation}
\end{example}

\begin{definition}%
[{\cite[Definitions~3.1]{BaL86}, \cite[Section~3]{PaS87}}]
A partial order is defined on quantum stopping times in the following
manner: if $S$ and $T$ are quantum stopping times then $S \le T$ if
and only if~$S\bigl( [ 0, t ] \bigr) \ge T\bigl( [ 0, t ] \bigr)$ for
all~$t \in [ 0, \infty ]$. The definition agrees with the classical
ordering in Examples~\ref{eg:classicalst} and~\ref{eg:poissonst}; in
the latter, $T_m \le T_n$ for all $m$, $n \in \Z_+$ such that
$m \le n$.
\end{definition}

\begin{theorem}\label{thm:qstorder}
If $S$ is a quantum stopping time and $s \in [ 0, \infty ]$ then
$S \wedge s$ is a quantum stopping time, where
\[
( S \wedge s )\bigl( [ 0, t ] \bigr) = \left\{ \begin{array}{cl}
 S\bigl( [ 0, t ] ) & \text{if } t < s, \\[1ex]
 I & \text{if } t \ge s
\end{array}\right.
\]
for all $t \in [ 0, \infty ]$. Furthermore, if $s$,
$t \in [ 0, \infty ]$ with $s \le t$ then
\[
S \wedge s \le S \wedge t \le S, \qquad S \wedge s \le s %
\qquad \text{and} \qquad ( S \wedge t ) \wedge s = S \wedge s.
\]
\end{theorem}
\begin{proof}
This is a straightforward exercise.
\end{proof}

\begin{remark}%
[{\Cf\cite[Lemma~2.3]{BaT87},\cite[Proposition~3.1]{PaS87}}]
Quantum stopping times $S \wedge T$ and~$S \vee T$ are
defined for any pair of quantum stopping times~$S$ and $T$ by setting
\begin{align*}
( S \wedge T )\bigl( [ 0, t ] \bigr) & = %
S\bigl( [ 0, t ] \bigr) \vee T\bigl( [ 0, t ] \bigr) \\[1ex]
\text{and} \quad %
( S \vee T )\bigl( [ 0, t ] \bigr) & = %
S\bigl( [ 0, t ] \bigr) \wedge T\bigl( [ 0, t ] \bigr) %
\qquad \text{for all } t \in [ 0, \infty ].
\end{align*}
It is straightforward to verify that
$t \mapsto ( S \wedge T )\bigl( [ 0, t ] \bigr)$ and
$t \mapsto ( S \vee T )\bigl( [ 0, t ] \bigr)$ are increasing.
Furthermore, $S \wedge T \le S \le S \vee T$ and
$S \wedge T \le T \le S \vee T$.

If $S\bigl( [ 0, t ] \bigr)$ commutes with $T\bigl( [ 0, t ] \bigr)$
then
\begin{align*}
( S \wedge T )\bigl( [ 0, t ] \bigr) & = %
S\bigl( [ 0, t ] \bigr) + T\bigl( [ 0, t ] \bigr) - %
S\bigl( [ 0, t ] \bigr) T\bigl( [ 0, t ] \bigr) \\[1ex]
\text{and} \quad ( S \vee T )\bigl( [ 0, t ] ) & = %
S\bigl( [ 0, t ] \bigr) T\bigl( [ 0, t ] \bigr),
\end{align*}
for all $t \in [ 0, \infty ]$.
\end{remark}

\section{Time projections}\label{sec:timeproj}

\begin{definition}
Let $E_0 \in \bop{\fock}{}$ be the orthogonal projection onto
$\C \evec{0}$, let~$E_t \in \bop{\fock}{}$ be the orthogonal
projection onto~$\fock_{t)} \otimes \evec{0|_{[ t, \infty )}}$,
considered as a subspace of~$\fock$, for all $t \in ( 0, \infty )$ and
let $E_\infty := I$. Then
\[
E_t \evec{f} = \evec{1_{[ 0, t )} f} %
\qquad \text{for all } t \in [ 0, \infty ] \text{ and } f \in \elltwo.
\]
\end{definition}

Given a quantum stopping time $S$, the \emph{time projection}
\[
E_S = \int_{[ 0, \infty ]} S( \rd s ) E_{s+}
= \int_{[ 0, \infty ]} E_{s+} S( \rd s ),
\]
where these integrals are strongly convergent limits of Riemann sums:
see Theorem~\ref{thm:expS}. Note that left, right and double stopping
$E = ( E_t )_{t \in [ 0, \infty ]}$ at $S$ produce the same result, as
observed in the Introduction.

\begin{definition}
A strictly increasing sequence $\pi = ( \pi_j )_{j \in \Z_+}$ with
$\pi_0 = 0$ and~$\lim\limits_{j \to \infty} \pi_j = \infty$ is said to
be a \emph{partition} of $\R_+$. A partition~$\pi'$ is a
\emph{refinement} of~$\pi$ if~$\pi$ is a subsequence of~$\pi'$.
\end{definition}

The following theorem dates back at least as far as
\cite[Theorem~2.3]{BaT90}.

\begin{theorem}\label{thm:expS}
Let $S$ be a quantum stopping time and let $\pi$ be a partition
of~$\R_+$. The series
\begin{equation}\label{eqn:es}
E_S^\pi := S\bigl( \{ 0 \} \bigr) E_0 + \sum_{j = 1}^\infty %
S\bigl( ( \pi_{j - 1}, \pi_j ] \bigr) E_{\pi_j} + %
S\bigl(\{ \infty \} \bigr)
\end{equation}
converges in the strong operator topology to an orthogonal
projection. If $\pi'$ is a refinement of $\pi$ then
$E_S^{\pi'} \le E_S^\pi$, so $E_S := \stlim_\pi E_S^\pi$ exists and is
an orthogonal projection such that $E_S \le E_S^\pi$ for all $\pi$.
Furthermore, it holds that
$E_S S\bigl( [ 0, t ] \bigr) = S\bigl( [ 0, t ] \bigr) E_S$ for all
$t \in [ 0, \infty ]$.
\end{theorem}
\begin{proof}
Note first that if $s$, $t \in \R_+$ are such that $s < t$ then
\[
S\bigl( ( s, t ] \bigr) = %
S\bigl( [ 0, t ] \bigr) - S\bigl( [ 0, s ] \bigr) \in %
\bop{\fock_{t)}}{} \otimes I_{[t} \qquad \text{and} \qquad %
E_t \in I_{t)} \otimes \bop{\fock_{[t}}{}.
\]
Hence if $m$, $n \in \Z_+$ are such that $m \le n$ then
\begin{align*}
\Bigl\| \sum_{j = m + 1}^n S\bigl( ( \pi_{j - 1}, \pi_j ] \bigr) %
E_{\pi_j} x \Bigr\|^2 & = \sum_{j = m}^n %
\| S\bigl( ( \pi_{j - 1}, \pi_j ] \bigr) E_{\pi_j} x \|^2 \\[1ex]
 & = \sum_{j = m + 1}^n %
\| E_{\pi_j} S\bigl( ( \pi_{j - 1}, \pi_j ] \bigr) x \|^2 \\[1ex]
 & \le \sum_{j = m + 1}^n %
\| S\bigl( ( \pi_{j - 1}, \pi_j ] \bigr) x \|^2 \\[1ex]
 & \le \| S\bigl( ( \pi_m, \infty ) \bigr) x \|^2
\end{align*}
for all $x \in \fock$, which gives the first claim; that $E_S^\pi$ is
an orthogonal projection is immediately verified. For the second, note
first that if $s$, $t \in \R_+$ are such that $0 < s \le t$ then
\[
( E_t - E_s )^2 = E_t^2 + E_s^2 - 2 E_s = E_t - E_s \in %
I_{s)} \otimes \bop{\fock_{[s}}{}.
\]
For all $j \in \Z_+$, let~$k_j \ge j$ be such that
$\pi'_{k_j} = \pi_j$ and let~$l_j \ge 1$ be such that
$\pi'_{k_j + l_j} = \pi_{j + 1}$. Then
\[
\langle x, ( E_S^\pi - E_S^{\pi'} ) x \rangle
= \sum_{j = 0}^\infty \sum_{l = 1}^{l_j} \langle x, %
 S\bigl( ( \pi'_{k_j + l - 1}, \pi'_{k_j + l} ] \bigr) %
 ( E_{\pi'_{k_j + l_j}} - E_{\pi'_{k_j + l}} ) x \rangle \ge 0
\]
for all $x \in \fock$, as required. The final claim is readily
verified.
\end{proof}

\begin{remark}
If the quantum stopping time $S$ corresponds to the deterministic time
$t \in [ 0, \infty ]$, so that $S\bigl( \{ t \} \bigr) = I$, then
$E_S = E_t$.
\end{remark}

\begin{example}[{\cite[Remark~3.6]{BaL86}}]
Let $M$ be a normal martingale with the chaotic-representation
property, such as a standard Brownian motion, and identity the Fock
space $\fock$ with $L^2\probsp$, as in Examples~\ref{eg:classicalst}
and~\ref{eg:poissonst}. If $\tau$ is a classical stopping time for $M$
then
\begin{equation}\label{eqn:classicaltp}
E_T \xi = \expn[ \xi | \salg_\tau ] %
\qquad \text{for all } \xi \in L^2\probsp,
\end{equation}
where $T$ is the quantum stopping time corresponding to~$\tau$, as in
Example~\ref{eg:classicalst}, and $\salg_\tau$ is the~$\sigma$-algebra
at the stopping time $\tau$. To see that (\ref{eqn:classicaltp})
holds, note first that, in this interpretation of Fock space, the
exponential vector $\evec{f}$ is a stochastic exponential and
satisfies the stochastic differential equation
\cite[Section~II.2]{Mey86}
\begin{equation}
\evec{f} = 1 + \int_0^\infty f( t ) \evec{1_{[ 0, t )} f} \std M_t %
\qquad \text{for all } f \in \elltwo,
\end{equation}
therefore
\begin{equation}
\expn[ \evec{f} | \salg_t ] = %
1 + \int_0^t f( s ) \evec{1_{[ 0, s )} f} \std M_s = %
\evec{1_{[ 0, t )} f} = E_t \evec{f}
\end{equation}
for all $t \in [ 0, \infty ]$.  Hence if $\tau$ takes values in the
set $\{ t_1 < \cdots < t_n \}$ then
\[
\expn[ \xi | \salg_\tau ] = %
\sum_{i = 1}^n 1_{\{ \tau = t_i \}} \expn[ \xi | \salg_{t_i} ] = %
\sum_{i = 1}^n 1_{\{ \tau = t_i \}} E_{t_i} \xi = E_T \xi
\]
for all $\xi \in L^2\probsp$. The general case now follows by
approximation: let~$\xi_t := \expn[ \xi | \salg_t ]$ for
all~$t \in [ 0, \infty ]$ and note that, given classical stopping
times $( \tau_n )_{n \in \N}$ such that~$\tau_n \to \tau$ almost
surely as $n \to \infty$,
\[
\expn[ \xi | \salg_{\tau_n} ] = \xi_{\tau_n} \to %
\xi_\tau = \expn[ \xi | \salg_\tau ] \qquad \text{as } n \to \infty,
\]
by optional sampling, almost surely and in $L^2\probsp$.
\end{example}

The following theorem has its origins in work of Meyer
\cite[equation~(12) on~p.74]{Mey87}; see also
\cite[Proposition~6]{AtS98}, \cite[Theorem~6.2]{AtC04} and
\cite[Theorem~2.5]{Coq06}. This representation shows that
$( E_{S \wedge s} )_{s \in \R_+}$ is a regular $\Omega$-martingale
closed by $E_S$; this is a very fruitful observation  when combined
with the quantum It\^o product formula. For the requisite details
in regard to quantum stochastic integration, see the appendix,
Section~\ref{sec:qsc}.

\begin{theorem}\label{thm:qstint}
Let $S$ be a quantum stopping time. Then
\begin{equation}\label{eqn:tpint}
E_S = I - \int_0^\infty I_\mul \otimes %
S\bigl( [ 0, s ] \bigr ) E_s \std \Lambda_s = %
E_0 + \int_0^\infty I_\mul \otimes %
S\bigl( ( s, \infty ] \bigr) E_s \std \Lambda_s.
\end{equation}
\end{theorem}
\begin{proof}
Note first that
\begin{align*}
\langle \evec{f}, %
\int_t^\infty I_\mul \otimes E_s \std \Lambda_s \, %
\evec{g} \rangle & = %
\int_t^\infty \langle f( s ), g( s ) \rangle \, %
\langle \evec{f}, E_s \evec{g} \rangle \std s \\[1ex]
 & = \int_t^\infty \frac{\rd}{\rd t} %
\langle \evec{f}, E_s \evec{g} \rangle \std s \\[1ex]
 & = \langle \evec{f}, ( I - E_t ) \evec{g} \rangle,
\end{align*}
so
\begin{align}\label{eqn:Eid}
E_t & = I - \int_t^\infty I_\mul \otimes E_s \std \Lambda_s \nonumber
\\[1ex]
 & = I - \int_0^\infty I_\mul \otimes E_s \std \Lambda_s + %
\int_0^t I_\mul \otimes E_s \std \Lambda_s = %
E_0 + \int_0^t I_\mul \otimes E_s \std \Lambda_s.
\end{align}
If $S\bigl( \{ 0 \} \bigr) = I$ then $E_S = E_0$ and the
identities hold as claimed. Now suppose that
$S\bigl( \{ 0 \} \bigr) = 0$ and let $\pi = ( \pi_j )_{j = 0}^\infty$
be a partition of $\R_+$. If $f$, $g \in \elltwo$ and $g$ has support
in $[ 0, \pi_n ]$ then
\begin{align*}
\langle \evec{f}, ( I - E^\pi_S ) \evec{g} \rangle & = %
\sum_{j = 1}^\infty \langle \evec{f}, %
S\bigl( ( \pi_{j - 1}, \pi_j ] \bigr) %
( I - E_{\pi_j} ) \evec{g} \rangle \\[1ex]
 & = \sum_{j = 1}^{n - 1} \sum_{k = j}^{n - 1} %
\langle \evec{f}, S\bigl( ( \pi_{j - 1}, \pi_j ] \bigr) %
\bigl( E_{\pi_{k + 1}} - E_{\pi_k} \bigr) \evec{g} \rangle \\[1ex]
 & = \sum_{k = 1}^{n - 1} \Bigl\langle \evec{f}, %
S\bigl( [ 0, \pi_k ] \bigr) %
\int_{\pi_k}^{\pi_{k + 1}} I_\mul \otimes E_s \std \Lambda_s \, %
\evec{f} \Bigr\rangle \\[1ex]
 & = \Bigl\langle \evec{f}, %
\Bigl( \int_0^\infty I_\mul \otimes S\bigl( [ 0, s ] \bigr) E_s %
\std \Lambda_s - R_\pi \Bigr) \evec{g} \Bigr\rangle,
\end{align*}
by Lemma~\ref{lem:qsi}, where
\[
R_\pi := \int_0^\infty \sum_{k = 0}^\infty %
1_{( \pi_k, \pi_{k + 1} ]}( s ) I_\mul \otimes %
S\bigl( ( \pi_k, s ] \bigr) E_s \std \Lambda_s.
\]
Finally, if $h \in \elltwo$ then
\begin{align*}
\| R_\pi \evec{h} \|^2 & \le \int_0^\infty \| \sum_{k = 0}^\infty %
1_{( \pi_k, \pi_{k + 1} ]}( s ) S\bigl( ( \pi_k, s ] \bigr) E_s %
\evec{h} \|^2 \| h( s ) \|^2 \std s \\[1ex]
 & = \int_0^\infty \sum_{k = 0}^\infty %
1_{( \pi_k, \pi_{k + 1} ]}( s ) \| S\bigl( ( \pi_k, s ] \bigr) E_s %
\evec{h} \|^2 \| h( s ) \|^2 \std s \to 0
\end{align*}
as $\pi$ is refined, by the dominated-convergence theorem. To see
this, note that if $s \in ( \pi_k, \pi_{k + 1} ]$ then
\begin{align*}
\| S\bigl( ( \pi_k, s ] \bigr) E_s \evec{h} \|^2 & = %
\| S\bigl( ( \pi_k, s ] \bigr) \evec{h} \|^2 \, %
\exp( -\| 1_{[ s, \infty )} h \|^2 ) \\[1ex]
 & \le \| S\bigl( ( \pi_k, s ] \bigr) \evec{h} \|^2 \\[1ex]
 & = \| S\bigl( [ 0, s ] \bigr) \evec{h} \|^2 - %
\| S\bigl( [ 0, \pi_k ] \bigr) \evec{h} \|^2 \to 0
\end{align*}
as $s \to \pi_k$, as long as $\pi_k$ is not a point of discontinuity
of $s \mapsto S\bigl( [ 0, s ] \bigr) \evec{h}$; furthermore,
\[
\sum_{k = 0}^\infty 1_{( \pi_k, \pi_{k + 1} ]}( s ) %
\| S\bigl( ( \pi_k, s ] \bigr) E_s \evec{h} \|^2 \, %
\| h( s ) \|^2 \le \| E_s \evec{h} \|^2 \, \| h( s ) \|^2
\]
for all $s \in \R_+$. The first identity is now established, and the
second may be obtained by writing
\[
I - E_0 = \!\int_0^\infty I_\mul \otimes E_s \std \Lambda_s = %
\int_0^\infty I_\mul \otimes %
S\bigl( [ 0, s ] \bigr) E_s \std \Lambda_s + %
\int_0^\infty I_\mul \otimes %
S\bigl( ( s, \infty ] \bigr) E_s \std \Lambda_s.
\qedhere
\]
\end{proof}

\begin{remark}
Let $S$ be a quantum stopping time. It follows from (\ref{eqn:tpint})
and Theorem~\ref{thm:qipf} that
\[
\| E_S x \|^2 = \| E_0 x \|^2 + %
\int_0^\infty \| ( I_\mul \otimes %
S\bigl( ( s, \infty ] \bigr) ) D_s x \|^2 \std s %
\qquad \text{for all } x \in \fock,
\]
where
\[
D : \fock \to L^2( \R_+; \mul \otimes \fock ); \ 
\big( D \evec{f} \bigr)( t ) = D_t \evec{f} := %
f( t ) \evec{1_{[ 0, t )} f}
\]
is the \emph{adapted gradient} \cite[Section~2.2]{AtL04}. When $S$ is
deterministic, this identity \cite[Proposition~2.3]{AtL04} is a key tool
for establishing the existence of quantum stochastic integrals,
particularly in the vacuum-adapted setting \cite{Blt01}.
\end{remark}

\begin{remark}\label{rem:qsconv}
Let $S$ and $T$ be quantum stopping times. Theorem~\ref{thm:qipf}
and~(\ref{eqn:tpint}) imply that
\[
\| ( E_S - E_T ) x \|^2 = \int_0^\infty \| \bigl( I_\mul \otimes %
( S\bigl( [ 0, s ] \bigr) - %
T\bigl( [ 0, s ] \bigr) ) \bigr) D_s x \|^2 \std s %
\qquad \text{for all } x \in \fock.
\]
It follows that the map $S \mapsto E_S$ is continuous when the set of
time projections is equipped with the strong operator topology and a
net of quantum stopping times $( S_\lambda )$ is defined to converge
to a spectral measure $S$ (which must then be a quantum stopping time)
if and only
if~$S_\lambda\bigl( [ 0, t ] \bigr) \to S\bigl( [ 0, t ] \bigr)$ in
the strong operator topology for all but a Lebesgue-null set of
points~$t$ in $[ 0, \infty ]$; this situation will be denoted by
``$S_\lambda \qsto S$''.

Any quantum stopping time $S$ is the limit, in this sense, of a
decreasing sequence of discrete quantum stopping times
$( S_n )_{n \in \N}$ \cite[Proposition~3.3]{PaS87},
\cite[Proposition~2.3]{BaW90}; a quantum stopping time $S$ is
\emph{discrete} if there exists a finite set
$A \subseteq [ 0, \infty ]$, the \emph{support} of $S$, such that
$S( A ) = I$ and $S( B ) \neq I$ if $B$ is any proper subset of~$A$.
Note that $S \wedge s \qsto S$ as~$s \to \infty$, so
$E_{S \wedge s} \to E_S$ in the strong operator topology.

In \cite{PaS87}, Parthasarathy and Sinha employ a weaker notion of
discreteness (allowing the support of $S$ to be countably infinite)
and a stronger notion of convergence (requiring
that~$S_\lambda\bigl( [ 0, t ] \bigr) \to S\bigl( [ 0, t ] \bigr)$ in
the strong operator topology for all $t \in [ 0, \infty ]$ such that
$S\bigl( \{ t \} \bigr) = 0$).
\end{remark}

\begin{example}{\cite[pp.507--508]{AtC04}, \cite[pp.323--324]{PaS87}}
Recall that the Boson Fock space~$\fock_+( \hilb )$ has a chaos
decomposition, so that
\[
\fock_+( \hilb ) = \bigoplus_{n = 0}^\infty \hilb^{\otimes_s n} %
\qquad \text{and} \qquad \evec{f} = %
\sum_{n = 0}^\infty \frac{1}{\sqrt{n!}} f^{\otimes n} %
\quad \text{for all } f \in \hilb,
\]
where $\hilb$ is any complex Hilbert space and $\otimes_s$ denotes the
symmetric tensor product.

Fix~$n \in \Z_+$ and, for all $t \in ( 0, \infty )$, let
$P_{n, t} \in \bop{\fock}{}$ be the orthogonal projection onto the
subspace
\[
\bigoplus_{j = 0}^n L^2\bigl( [ 0, t ); \mul \bigr)^{\otimes_s j} %
\otimes \fock_{[t} \subseteq \fock_{t)} \otimes \fock_{[t} = \fock;
\]
let $P_{n, 0} = I$, let $P_{n, \infty} = 0$ and let $P_n$ be the
orthogonal projection onto
$\bigoplus_{j = 0}^n \elltwo^{\otimes_s j}$. Note
that~$P_{n, s} \ge P_{n, t}$ for all $s$, $t \in [ 0, \infty ]$ such
that $s \le t$, so setting
\begin{equation}
S_n\bigl( [ 0, t ] \bigr) := P_{n, t}^\perp = I - P_{n, t} %
\qquad \text{for all } t \in [ 0, \infty ]
\end{equation}
defines a stopping time. Furthermore, if $t \in \R_+$ and
$f \in \elltwo$ then
\[
S_n\bigl( ( t, \infty ] \bigr) E_t \evec{f} = %
P_{n, t} \evec{1_{[ 0, t )} f} = P_n E_t \evec{f},
\]
so
\begin{align}
\langle \evec{f}, E_{S_n} \evec{g} \rangle & = %
1 + \int_0^t \langle \evec{f}, P_n E_s \evec{g} \rangle \, %
\langle f( s ), g( s ) \rangle \std s \nonumber \\[1ex]
 & = 1 + \sum_{k = 0}^n \frac{1}{k!} \int_0^t \Bigl( %
\int_0^s \langle f( r ), g( r ) \rangle \std r \Bigr)^k %
\langle f( s ), g( s ) \rangle \std s \nonumber \\[1ex]
 & = 1 + \sum_{k = 0}^n \frac{1}{( k + 1 )!} \Bigl( %
\int_0^t \langle f( s ), g( s ) \rangle \std s \Bigr)^{k + 1} %
\nonumber \\[1ex]
 & = \langle \evec{f}, P_{n + 1} \evec{g} \rangle
\end{align}
for all $f$, $g \in \elltwo$. Thus $E_{S_n} = P_{n + 1}$.

Finally, note that $P_{n, t} \le P_{n + 1, t}$ for all
$t \in [ 0, \infty ]$, so $S_n \le S_{n + 1}$ for all~$n \in \Z_+$.
\end{example}

\begin{definition}[{\cite[Section~5]{PaS87}}]
If $S$ is a quantum stopping time $S$ then the
\emph{pre-$S$ space}~$\fock_{S)}$ is the range of the time projection
$E_S$; thus $\fock_{S)} := E_S( \fock )$.
\end{definition}

\begin{theorem}%
[{\Cf\cite[Theorem~3.7]{BaT87},\cite[Theorem~3.5]{BaW90}}]%
\label{thm:exporder}
Let $
S$ and $T$ be quantum stopping times.
\begin{mylist}
\item[\tu{(i)}] If $S \le T$ then $E_S \le E_T$ and
$\fock_{S)} \subseteq \fock_{T)}$.
\item[\tu{(ii)}] The time projections
$E_{S \wedge T} = E_S \wedge E_T$ and $E_{S \vee T} = E_S \vee E_T$.
\item[\tu{(iii)}] If $S\bigl( [ 0, t ] \bigr)$ and
$T\bigl( [ 0, t ] \bigr)$ commute for all $t \in \R_+$ then so
do~$E_S$ and $E_T$.
\item[\tu{(iv)}] If $s \in \R_+$ then
$E_S E_s = E_{S \wedge s} = E_s E_S$.
\end{mylist}
\end{theorem}
\begin{proof}
Some of these claims may be established by working from the
definitions, but Theorem~\ref{thm:qstint} and the quantum It\^o
product formula, Theorem~\ref{thm:qipf}, provide a slicker means of
obtaining them in the Fock-space context.

(i) As
$S\bigl( [ 0, s ] \bigr) T\bigl( [ 0, s ] \bigr) = %
T\bigl( [ 0, s ] \bigr)$
for all $s \in [ 0, \infty ]$, it follows that
\begin{align*}
E_S E_T & = \Bigl( I - \int_0^\infty I_\mul \otimes %
S\bigl( [ 0, s ] \bigr) E_s \std \Lambda_s \Bigr) %
\Bigl( I - \int_0^\infty I_\mul \otimes %
T\bigl( [ 0, s ] \bigr) E_s \std \Lambda_s \Bigr) \\[1ex]
 & = I - \int_0^\infty I_\mul \otimes %
( S\bigl( [ 0, s ] \bigr) E_s + T\bigl( [ 0, s ] \bigr) E_s - %
S\bigl( [ 0, s ] \bigr) T\bigl( [ 0, s ] \bigr) E_s ) %
\std \Lambda_s \\[1ex]
 & = I - \int_0^\infty I_\mul \otimes %
S\bigl( [ 0, s ] \bigr) E_s \std \Lambda_s \\[1ex]
 & = E_S.
\end{align*}
(ii) By the quantum It\^o product formula and von~Neumann's method
of alternating projections,
\begin{align*}
( E_S^\perp E_T^\perp )^n & = \int_0^\infty I_\mul \otimes \Bigl( %
( S\bigl( [ 0, s ] \bigr) T\bigl( [ 0, s ] \bigr) \Bigr)^n %
E_s \std \Lambda_s \\
 & \to \int_0^\infty I_\mul \otimes %
( S \wedge T )\bigl( [ 0, s ] \bigr) E_s \std \Lambda_s = %
E_{S \vee T}^\perp
\end{align*}
as $n \to \infty$, so
$E_{S \vee T} = ( E_S^\perp \wedge E_T^\perp )^\perp = E_S \vee E_T$.

For the second identity, note first that $E_S - E_0$ and $E_T - E_0$
are orthogonal projections, by (i), and, as $n \to \infty$,
\begin{align*}
\bigl( ( E_S - E_0 ) ( E_T - E_0 ) \bigr)^n & = %
\int_0^\infty I_\mul \otimes %
\Bigl( S\bigl( ( s, \infty ] \bigr) %
T\bigl( ( s, \infty ] \bigr) \Bigr)^n E_s \std \Lambda_s \\[1ex]
 & \to \int_0^\infty I_\mul \otimes %
( S \wedge T )\bigl( ( s, \infty ] \bigr) E_s \std \Lambda_s \\[1ex]
 & = E_{S \wedge T} - E_0.
\end{align*}
However, as $E_S E_0 = E_0 E_S = E_0 = E_0 E_T = E_T E_0$, we also
have that
\[
\bigl( ( E_S - E_0 ) ( E_T - E_0 ) \bigr)^n = ( E_S E_T - E_0 )^n = %
( E_S E_T )^n - E_0 \to E_S \wedge E_T - E_0,
\]
which gives the result.

(iii) As
$S\bigl( [ 0, s ] \bigr) T\bigl( [ 0, s ] \bigr) = %
T\bigl( [ 0, s ] \bigr) S\bigl( [ 0, s ] \bigr)$
for all $s \in \R_+$, it follows that
\[
E_S E_T = I - \int_0^\infty I_\mul \otimes %
\Bigl( S\bigl( [ 0, s ] \bigr) + T\bigl( [ 0, s ] \bigr) - %
S\bigl( [ 0, s ] \bigr) T\bigl( [ 0, s ] \bigr) \Bigr) E_s %
\std \Lambda_s = E_T E_S.
\]
(iv) By (ii) and (iii), as $S$ and $s$ commute as required,
so $E_S \wedge E_s = E_S E_s$.
\end{proof}

\begin{proposition}\label{prp:Sdecomp}
For all $s \in [ 0, \infty ]$, it holds that
\[
E_{S \wedge s} = S\bigl( [ 0, s ] \bigr) E_S + %
S\bigl( ( s, \infty ] \bigr) E_s.
\]
\end{proposition}
\begin{proof}
Without loss of generality, let $s \in \R_+$ and let $\pi$ be a
partition of $\R_+$ with $s = \pi_n$ for some $n \in \Z_+$. Then
\begin{align*}
E^\pi_S E_s & = S\bigl( \{ 0 \} \bigr) E_0 + %
\sum_{j = 1}^n S\bigl( ( \pi_{j - 1}, \pi_j ] \bigr) E_{\pi_j} + %
\sum_{j = n + 1}^\infty S\bigl( ( \pi_{j - 1}, \pi_j ] \bigr) E_s \\
 & \hspace{24em} + S\bigl( \{ \infty \} \bigr) E_s \\[1ex]
 & = S\bigl( [ 0, s ] \bigr) E_S^\pi + %
S\bigl( ( s, \infty ] \bigr) E_s.
\end{align*}
The claim now follows by refining $\pi$, since
$E_S E_s = E_{S \wedge s}$ by Theorem~\ref{thm:exporder}(iii).
\end{proof}

\section{The stopping algebras}\label{sec:algebras}

\begin{definition}\label{def:stopop}
Given $Z \in \bop{\fock}{}$ and a quantum stopping time $S$, let
\[
Z_{\wc{S}} := E_S Z E_S \in \bop{\fock}{}.
\]
Note that $E_S Z_{\wc{S}} E_S = Z_{\wc{S}} E_S$, and so $Z_{\wc{S}}$
maps $\fock_{S)}$ to itself. Remark~\ref{rem:qsconv} implies that the
mapping~$( Z, \, S ) \mapsto Z_{\wc{S}}$ is jointly continuous on the
product of any bounded subset of $\bop{\fock}{}$ with the collection
of all quantum stopping times, when $\bop{\fock}{}$ is equipped with
the strong operator topology and a net $( S_\lambda )$ of quantum
stopping times converges to the quantum stopping time $S$ if and only
if~$S_\lambda \qsto S$, as defined in Remark~\ref{rem:qsconv}.
\end{definition}

\begin{proposition}\label{prp:vaccondexp}
The map $Z \mapsto Z_{\wc{S}}$ is a conditional expectation from
$\bop{\fock}{}$ onto the norm-closed $*$-subalgebra
$\bopp_{\wc{S}} := \{ E_S Z E_S : Z \in \bop{\fock}{} \}$ that
preserves the \emph{vacuum state}
\begin{equation}\label{eqn:vacuumstate}
\expn_\Vac : \bop{\fock}{} \to \C; \ %
Z \mapsto \langle \evec{0}, Z \evec{0} \rangle.
\end{equation}
\end{proposition}
\begin{proof}
This is a straightforward exercise.
\end{proof}

\begin{remark}
The collection of stopped algebras
\[
\{ \bopp_{\wc{S}} : S \text{ is a quantum stopping time} \}
\]
has the following properties.
\begin{mylist}
\item[(i)] If $S$ is a quantum stopping time and
$Z \in \bopp_{\wc{S}}$ then $E_S Z E_S = Z E_S$, so $Z$ preserves the
pre-$S$ space $\fock_{S)}$.
\item[(ii)] If the quantum stopping times $S$ and $T$ satisfy
$S \le T$ then $\bopp_{\wc{S}} \subseteq \bopp_{\wc{T}}$, by
Theorem~\ref{thm:exporder}(i).
\item[(iii)] For deterministic stopping times,
\[
\bopp_{\wc{0}} = \im P^\Vac_{[0}, \qquad %
\bopp_{\wc{t}} = \bop{\fock_{t)}}{} \otimes P^\Vac_{[t} %
\qquad \text{and} \qquad \bopp_{\wc{\infty}} = \bop{\fock}{}
\]
for all $t \in ( 0, \infty )$, where
$P^\Vac_{[s} \in \bop{\fock_{[s}}{}$ is the orthogonal projection onto
the vacuum subspace $\C \evec{0|_{[ s, \infty )}}$ for all
$s \in \R_+$.
\end{mylist}
As noted by Coquio, this is impossible if we work instead in the
identity-adapted setting with the natural analogue of (iii): see
\cite[Proposition~2.2]{Coq06}.
\end{remark}

\section{Processes and martingales}\label{sec:processes}

\begin{definition}\label{def:mg}
A \emph{process} is a family
$X = ( X_t )_{t \in \R_+} \subseteq \bop{\fock}{}$. (We have no need
to impose any measurability or adaptedness conditions at this point.)
Two processes $X$ and $Y$ are equal if and only if $X_t = Y_t$ for all
$t \in \R_+$. The set of processes is a complex associative
$*$-algebra, where addition, multiplication and the adjoint are
defined pointwise.

Given $Z \in \bop{\fock}{}$, let $\wc{\pi}( Z )$ and $\wh{\pi}( Z )$
be the processes with initial values
\[
\wc{\pi}( Z )_0 = \expn_\Vac[ Z ] \, E_0 \qquad \text{and} \qquad %
\wh{\pi}( Z )_0 = \expn_\Vac[ Z ] \, I,
\]
and such that
\[
\wc{\pi}( Z )_t = E_t Z E_t \qquad \text{and} \qquad %
\wh{\pi}( Z )_t = E_t Z |_{\fock_{t)}} \otimes I_{[t} %
\qquad \text{for all } t \in ( 0, \infty ).
\]
Note that $t \mapsto \wc{\pi}( Z )_t$ and $t \mapsto \wh{\pi}( Z )_t$
are uniformly bounded and continuous on~$[ 0, \infty ]$ in the strong
operator topology, where
$\wc{\pi}( Z )_\infty = \wh{\pi}( Z )_\infty := Z$, for
any~$Z \in \bop{\fock}{}$. Note also that the maps
$Z \mapsto \wc{\pi}( Z )$ and $Z \mapsto \wh{\pi}( Z )$ are
$*$-algebra homomorphisms. We extend these definitions from operators
to processes by setting
\[
\wc{\pi}( X )_t := \wc{\pi}( X_t )_t \qquad \text{and} \qquad %
\wh{\pi}( X )_t = \wh{\pi}( X_t )_t \qquad %
\text{for all } t \in [ 0, \infty ],
\]
where $X$ is an arbitrary process; if not otherwise defined, we let
$X_\infty := 0$. For convenience, we will also consider processes
indexed by $[ 0, \infty ]$.

A process $X$ is \emph{adapted} if $X_t E_t = E_t X_t$ for all
$t \in \R_+$.

A process $X$ is \emph{vacuum adapted} if $X = \wc{\pi}( X )$; this is
equivalent to the requirement that $X_t E_t = X_t = E_t X_t$ for
all~$t \in \R_+$. Note that $\wc{\pi}( Z )$ is a vacuum-adapted
process for any $Z \in \bop{\fock}{}$.

A process $X$ is \emph{identity adapted} if $X = \wh{\pi}( X )$. Note
that $\wh{\pi}( Z )$ is an identity-adapted process for any
$Z \in \bop{\fock}{}$.

Vacuum-adapted and identity-adapted processes are adapted, and the
sets of adapted processes, vacuum-adapted processes and
identity-adapted processes are $*$-subalgebras of the $*$-algebra of
processes.

The process $M$ is a \emph{martingale} if it is adapted and
$E_s M_t E_s = M_s E_s$ for all $s$, $t \in \R_+$ with $s \le t$.

The martingale $M$ is \emph{closed} if
$\stlim\limits_{t \to \infty} M_t = M_\infty$ for some
$M_\infty \in \bop{\fock}{}$, where ``$\stlim$'' denotes the limit in
the strong operator topology.

The sets of martingales and closed martingales are subspaces of the
algebra of adapted processes.
\end{definition}

\begin{proposition}\label{prp:closedvacmg}
A process $X$ is a vacuum-adapted martingale closed by~$X_\infty$ if
and only if~$X = \wc{\pi}( X_\infty )$.
\end{proposition}
\begin{proof}
Suppose $X = \wc{\pi}( X_\infty )$. Then $X$ is vacuum adapted, and if
$s$, $t \in \R_+$ are such that $s \le t$ then
\[
E_s X_t E_s = E_s E_t X_\infty E_t E_s = E_s X_\infty E_s = %
X_s = X_s E_s,
\]
so $X$ is a martingale. Furthermore,
$X_t = %
\wc{\pi}( X_\infty )_t \to \wc{\pi}( X_\infty )_\infty = X_\infty$
in the strong operator topology as $t \to \infty$, so $X$ is closed by
$X_\infty$.

Conversely, if $X$ is a vacuum-adapted martingale and
$X_\infty = \stlim\limits_{t \to \infty} X_t$ then, for all
$s \in \R_+$,
\[
\wc{\pi}( X_\infty )_s = E_s X_\infty E_s = %
\stlim_{t \to \infty} E_s X_t E_s = \stlim_{t \to \infty} X_s E_s = %
X_s E_s = X_s.
\qedhere
\]
\end{proof}

\begin{proposition}\label{prp:closedidmg}
A process $X$ is an identity-adapted martingale closed by~$X_\infty$
if and only if~$X = \wh{\pi}( X_\infty )$.
\end{proposition}
\begin{proof}
Suppose $X = \wh{\pi}( X_\infty )$. Then $X$ is identity adapted, and
if $s$, $t \in \R_+$ are such that $s \le t$ then
\[
E_s X_t E_s = E_s \wh{\pi}( X_\infty )_t E_s = E_s X_\infty E_s = %
\wh{\pi}( X_\infty )_s E_s,
\]
so $X$ is a martingale. Furthermore,
$X_t = %
\wh{\pi}( X_\infty )_t \to \wh{\pi}( X_\infty )_\infty = X_\infty$
in the strong operator topology as $t \to \infty$, so $X$ is closed by
$X_\infty$.

Conversely, if $X$ is an identity-adapted martingale and
$X_\infty = \stlim_{t \to \infty} X_t$ then, for all~$s \in \R_+$,
\[
\wh{\pi}( X_\infty )_s E_s = %
E_s X_\infty E_s = \stlim_{t \to \infty} E_s X_t E_s = X_s E_s
\]
and
\[
\wh{\pi}( X_\infty )_s = %
\wh{\pi}( \wh{\pi}( X_\infty )_s E_s )_s = \wh{\pi}( X_s E_s )_s = %
\wh{\pi}( X_s )_s = X_s.
\qedhere
\]
\end{proof}

\begin{proposition}\label{prp:vmg}
If $Z \in \bop{\fock}{}$ and $S$ is a quantum stopping time then
the process $Z^{\wc{S}} := ( Z_{\wc{S \wedge t}} )_{t \in \R_+}$ is a
vacuum-adapted martingale closed by~$Z_{\wc{S}}$.
\end{proposition}
\begin{proof}
If $t \in \R_+$ then Theorem~\ref{thm:exporder}(ii) implies that
\[
E_t Z_{\wc{S}} E_t = E_t E_S Z E_S E_t = E_{S \wedge t} Z E_{S \wedge t} = %
Z^{\wc{S}}_t,
\]
so the result follows from Proposition~\ref{prp:closedvacmg}.
\end{proof}

\begin{remark}
The process $Z^{\wc{S}}$ remains a vacuum-adapted martingale if
$( E_t )_{t \in \R_+}$ is replaced
by~$( E_{S \wedge t} )_{t \in \R_+}$ in Definition~\ref{def:mg}, as
the analogue of Proposition~\ref{prp:closedvacmg} holds.
\end{remark}

\section{Stopping processes at discrete times}\label{sec:discrete}

\begin{definition}
A quantum stopping time $T$ is said to be \emph{discrete} if there
exists a finite set of
times~$\{ t_1 < \ldots < t_n \} \subseteq [ 0, \infty ]$,
called the \emph{support} of $T$, such that
$T\bigl( \{ t_i \} \bigr) \neq 0$ for $i = 1$, \ldots, $n$ and
$T\bigl( \{ t_1, \ldots, t_n \} \bigr) = I$. Note that
\[
T\bigl( [ 0, t ] \bigr) = \left\{ \begin{array}{cl}
 0 & \text{if } t < t_1, \\[1ex]
 T\bigl( \{ t_1, \ldots, t_m \} ) & \text{if } %
t \in [ t_m, t_{m + 1} ) \quad ( m = 1, \ldots, n - 1 ), \\[1ex]
 I & \text{if } t \ge t_n
\end{array}\right.
\]
for all $t \in [ 0, \infty ]$. In particular, $t_1 \le T \le t_n$.
\end{definition}

\begin{definition}\label{def:discrete_stop}
If $X$ is a process and $T$ is a discrete quantum stopping time with
support $\{ t_1, \ldots, t_n \}$ then Coquio
\cite[Definition~3.2(2)]{Coq06} stops $X$ at $T$ by setting
\begin{align*}
X_{\wh{T}} & := %
\sum_{i, j = 1}^n \wh{\pi}( E_{t_i} )_{t_i \vee t_j} T_i %
\wh{\pi}( X )_{t_i \vee t_j} T_j %
\wh{\pi}( E_{t_j} )_{t_i \vee t_j} \\[1ex]
 & \hphantom{:}= \sum_{i, j = 1}^n \wh{\pi}\bigl( %
E_{t_i} T_i X_{t_i \vee t_j} T_j E_{t_j} \bigr)_{t_i \vee t_j},
\end{align*}
where $T_i := T\bigl( \{ t_i \} \bigr)$ for $i = 1$, \ldots
$n$. Recall that $X_\infty := 0$ if it is not otherwise defined.

Note that if $t \in [ 0, \infty ]$ then
$X_{\wh{t}} = \wh{\pi}( X )_t$.  Furthermore, $E_{\wh{T}} = I$ and,
considering $I$ as a constant process, $I_{\wh{T}} = I$.
\end{definition}

\begin{remark}
Coquio only considers identity-adapted processes, for which her
definition agrees with Definition~\ref{def:discrete_stop} above. She
prefers the notation $M_T( Z )$ when stopping an operator; as we do
not require processes to be adapted, so may identify operators with
constant processes, our choice of notation seems more convenient.
\end{remark}

\begin{definition}\label{def:discstop}
Let $X$ be a process and let $T$ be a discrete quantum stopping time
with support $\{ t_1 < \ldots < t_n \}$. The result of applying
vacuum-adapted stopping to~$X$ at~$T$ is
\[
X_{\wc{T}} := %
\sum_{i, j = 1}^n E_{t_i} T_i X_{t_i \vee t_j} T_j E_{t_j} = %
\sum_{i, j = 1}^n T_i E_{t_i} X_{t_i \vee t_j} E_{t_j} T_j,
\]
where $T_i := T\bigl( \{ t_i \} \bigr)$ for $i = 1$, \ldots, $n$;
again, let $X_\infty := 0$ if necessary. This is analogous to Coquio's
definition for the identity-adapted case.

Note that if $t \in [ 0, \infty ]$ then
$X_{\wc{t}} = \wc{\pi}( X )_t$. Furthermore, if the process $X$ has
the constant value~$Z$ then $X_{\wc{T}} = E_T Z E_T = Z_{\wc{T}}$ in
the sense of Definition~\ref{def:stopop}; in particular,
$I_{\wc{T}} = E_T$. It also follows from the definition that
$E_{\wc{T}} = E_T$.
\end{definition}

\begin{proposition}\label{prp:twostops}
If $X$ is a process and $T$ is a discrete quantum stopping time then
\[
( X_{\wc{T}} )_{\wc{T}} = X_{\wc{T}}, \qquad %
( X_{\wh{T}} )_{\wc{T}} = X_{\wc{T}}, \qquad %
( X_{\wc{T}} )_{\wh{T}} = X_{\wh{T}} \qquad \text{and} \qquad %
( X_{\wh{T}} )_{\wh{T}} = X_{\wh{T}},
\]
where $X_{\wh{T}}$ and $X_{\wc{T}}$ are stopped by regarding them as
constant processes.
\end{proposition}
\begin{proof}
This follows because
$E_r \wh{\pi}( E_r )_s = E_r = \wh{\pi}( E_r )_s E_r$ for
all~$r$, $s \in [ 0, \infty ]$ such that $r \le s$, and
$E_r \wh{\pi}( X_{r \vee s} )_{r \vee s} E_s = E_r X_{r \vee s} E_s$
for all $r$, $s \in [ 0, \infty ]$.
\end{proof}

\begin{proposition}\label{prp:procSad}
If $T$ is a discrete quantum stopping time and $X$ is a process then
\[
T\bigl( [ 0, t ] \bigr) X_{\wc{T}} T\bigl( [ 0, t ] \bigr) = %
T\bigl( [ 0, t ] \bigr) X_{\wc{T \wedge t}} T\bigl( [ 0, t ] \bigr),
\]
whereas
\[
X_{\wh{T}} E_T = E_T X_{\wh{T}} = X_{\wc{T}} %
\quad \text{and} \quad %
T\bigl( [ 0, t ] \bigr) X_{\wh{T}} T\bigl( [ 0, t ] \bigr) = %
T\bigl( [ 0, t ] \bigr) X_{\wh{T \wedge t}} T\bigl( [ 0, t ] \bigr),
\]
for all $t \in [ 0, \infty ]$.
\end{proposition}
\begin{proof}
Without loss of generality, suppose that $t \in [ t_k, t_{k +1} )$,
where $t_0 := 0$ and~$t_{n + 1} := \infty$. Since
$T\bigl( [ 0, t ] \bigr)( T \wedge t )\bigl( \{ t \} \bigr) = %
T\bigl( \{ t \} \bigr)$
and the support of~$T \wedge t$ has maximum element~$t$, it follows
that
\[
T\bigl( [ 0, t ] \bigr) X_{\wc{T}} T\bigl( [ 0, t ] \bigr) = %
\sum_{i, j = 1}^k E_{t_i} T_i X_{t_i \vee t_j} T_j E_{t_j} = %
T\bigl( [ 0, t ] \bigr) X_{\wc{T \wedge t}} T\bigl( [ 0, t ] \bigr).
\]
The other identities are contained in
\cite[Properties~3.3(1--2)]{Coq06}.
\end{proof}

\begin{proposition}\label{prp:exe}
If $T$ is a discrete quantum stopping time and $X$ is a process then
\[
\wc{\pi}( X_{\wc{T \wedge t}} )_t = X_{\wc{T \wedge t}} %
\qquad \text{and} \qquad
\wh{\pi}( X_{\wh{T \wedge t}} )_t = X_{\wh{T \wedge t}} %
\qquad \text{for all } t \in [ 0, \infty ],
\]
so the processes $X^{\wc{T}} := ( X_{\wc{T \wedge t}} )_{t \in \R_+}$
and $X^{\wh{T}} := ( X_{\wh{T \wedge t}} )_{t \in \R_+}$ are vacuum
adapted and identity adapted, respectively.
\end{proposition}
\begin{proof}
Since $T \wedge t$ is discrete and $T \wedge t \le t$, by
Theorem~\ref{thm:qstorder}, it follows from
the first part of Proposition~\ref{prp:procSad} and
Theorem~\ref{thm:exporder}(i) that
\[
E_t X_{\wc{T \wedge t}} E_t = %
E_t E_{T \wedge t} X_{\wc{T \wedge t}} E_{T \wedge t} E_t = %
E_{T \wedge t} X_{\wc{T \wedge t}} E_{T \wedge t} = %
X_{\wc{T \wedge t}}.
\]
The second claim is contained in \cite[Properties~3.3(2)]{Coq06}.
\end{proof}

\begin{remark}
The definitions of $X^{\wc{T}}$ and~$Z^{\wc{S}}$ in
Propositions~\ref{prp:exe} and~\ref{prp:vmg} are consistent: they
agree when when $S = T$ is discrete and $Z$ is regarded as a constant
process~$X$, by the penultimate remark in
Definition~\ref{def:discstop}.
\end{remark}

\begin{lemma}\label{lem:discmg}
If $M$ is a martingale and $T$ is a discrete quantum stopping time
with support $\{ t_1 < \cdots < t_n \}$ then
$M_{\wc{T}} = E_T M_t E_T$ for all $t \in [ t_n, \infty )$.
\end{lemma}
\begin{proof}
With the notation of Definition~\ref{def:discstop}, if
$t \in [ t_n, \infty )$ then
\[
M_{\wc{T}} = \sum_{i, j = 1}^n %
T_i E_{t_i} E_{t_i \vee t_j} M_t E_{t_i \vee t_j} E_{t_j} T_j = %
\sum_{i, j = 1}^n T_i E_{t_i} M_t E_{t_j} T_j = E_T M_t E_T.
\qedhere
\]
\end{proof}

\begin{theorem}\label{thm:mgchar}
Let $X$ be an adapted process. Then $X$ is a martingale if and only
if~$\expn_\Vac\bigl[ X_{\wc{T}} \bigr] = \expn_\Vac\bigl[ X_0 \bigr]$
for every discrete quantum stopping time $T$, where~$\expn_\Vac$ is
the vacuum state \tu{(\ref{eqn:vacuumstate})}.
\end{theorem}
\begin{proof}
We follow the proof of \cite[Proposition~3.10]{Coq06}. If $X$ is a
martingale and the discrete quantum stopping time $T$ has support
$\{ t_1 < \cdots < t_n \}$ then Lemma~\ref{lem:discmg} implies
that
\[
\expn_\Vac\bigl[ X_{\wc{T}} \bigr] = %
\langle \evec{0}, E_0 E_T X_{t_n} E_T E_0 \evec{0} \rangle = %
\langle \evec{0}, E_0 X_{t_n} E_0 \evec{0} \rangle = %
\expn_\Vac[ X_0 ];
\]
note that $T \wedge 0 = 0$ and $E_0 E_T = E_0 = E_T E_0$, by
Theorem~\ref{thm:exporder}(ii).

Conversely, let $T$ have support $\{ s < t \} \subseteq \R_+$, let
$T\bigl( \{ s \} \bigr) = P$ and note that
\begin{align*}
\expn_\Vac\bigl[ X_{\wc{T}} \bigr] & = %
\expn_\Vac[ P X_s P ] + %
\expn_\Vac\bigl[ P X_t P^\perp \bigr] + %
\expn_\Vac\bigl[ P^\perp X_t P \bigr] + %
\expn_\Vac\bigl[ P^\perp X_t P^\perp \bigr] \\[1ex]
 & = \expn_\Vac[ P ( X_s - X_t )P + X_t ],
\end{align*}
so if
$\expn_\Vac\bigl[ X_{\wc{T}} \bigr] = \expn_\Vac[ X_0 ] = %
\expn_\Vac\bigl[ X_{\wc{t}} \bigr] = \expn_\Vac[ X_t ]$
then
\[
\langle P \evec{0}, ( X_t - X_s ) P \evec{0} \rangle = 0
\]
for any orthogonal projection
$P \in \bop{\fock_{s)}}{} \otimes I_{[s}$, so for any
$P \in \bop{\fock_{s)}}{} \otimes I_{[s}$. It follows that
$E_s ( X_t - X_s ) E_s = 0$, as required.
\end{proof}

\begin{remark}
Since
$\expn_\Vac\bigl[ X_{\wc{T}} \bigr] = %
\expn_\Vac\bigl[ X_{\wh{T}} \bigr]$
for any process $X$ and any discrete quantum stopping time $T$, the
identity-adapted version of Theorem~\ref{thm:mgchar} holds: Coquio's
work contains a similar result \cite[Proposition~3.10]{Coq06}.
\end{remark}

The following result expresses the relationship between the vacuum and
identity-adapted martingales closed by the same operator;
\cf\cite[Lemma~3.6]{Coq06}. It is provides the key to stopping
processes for more general times.

\begin{lemma}\label{lem:Mint}
If $X$ is a process and $t \in [ 0, \infty ]$ then
\[
\wh{\pi}( X )_t = \wc{\pi}( X )_t + \int_t^\infty %
I_\mul \otimes \wh{\pi}( \wc{\pi}( X )_t )_s \std \Lambda_s = %
\wc{\pi}( X )_t + \int_t^\infty I_\mul \otimes %
\wc{\pi}( \wh{\pi}( X )_t )_s \std \Lambda_s.
\]
\end{lemma}
\begin{proof}
Applying Theorem~\ref{thm:switch} to the first identity in
(\ref{eqn:Eid}), it follows that
\[
I - E_t = \wh{\pi}( I - E_t )_\infty = %
\int_t^\infty I_\mul \otimes %
( I - \wh{\pi}( I - E_t )_s ) \std \Lambda_s = %
\int_t^\infty I_\mul \otimes \wh{\pi}( E_t )_s \std \Lambda_s.
\]
Hence, by Lemma~\ref{lem:qsi},
\[
\wh{\pi}( X )_t = \wh{\pi}( X )_t \Bigl( E_t + %
\int_t^\infty I_\mul \otimes \wh{\pi}( E_t )_s \std \Lambda_s \Bigr) %
= \wc{\pi}( X )_t + \int_t^\infty I_\mul \otimes %
\wh{\pi}( \wc{\pi}( X )_t )_s \std \Lambda_s.
\]
Similar working, but using the first identity in (\ref{eqn:Eid})
directly, gives the second claim.
\end{proof}

\begin{remark}
It does not follow from Lemma~\ref{lem:Mint} that
$\wh{\pi}( \wc{\pi}( X )_t )_s = \wc{\pi}( \wh{\pi}( X )_t )_s$ for
almost all~$s \ge t$. This is because the integrands are adapted in
difference senses, so their difference is neither identity adapted nor
vacuum adapted, in general; thus Remark~\ref{rem:independence} and
Lemma~\ref{lem:qsiestimate} cannot be applied. The same situation
occurs repeatedly below; for example, see
Theorem~\ref{thm:idclosedmg}.
\end{remark}

The following theorem extends Lemma~\ref{lem:Mint} from deterministic
to discrete quantum stopping times. If it is known how to stop a class
of processes for one form of adaptedness, this result provides
stopping of that class for the other form. The identity-adapted
version for closed martingales was obtained by Coquio
\cite[Proof of Theorem~3.5]{Coq06}.

\begin{theorem}\label{thm:idclosedmg}
If $X$ is a process and $T$ is a discrete quantum stopping time then
\begin{align}
X_{\wh{T}} & = X_{\wc{T}} + \int_0^\infty I_\mul \otimes %
T\bigl( [ 0, s ] \bigr) \wh{\pi}\bigl( X_{\wc{T}} \bigr)_s %
T\bigl( [ 0, s ] \bigr) \std \Lambda_s \label{eqn:idclosedmg} \\[1ex]
 & = X_{\wc{T}} + \int_0^\infty I_\mul \otimes %
T\bigl( [ 0, s ] \bigr) \wc{\pi}( X_{\wh{T}} )_s %
T\bigl( [ 0, s ] \bigr) \std \Lambda_s.
\label{eqn:vacclosedmg}
\end{align}
\end{theorem}
\begin{proof}
If $T$ has support $\{ t_1 < \cdots < t_n \}$ and
$T_i = T\bigl( \{ t_i \} \bigr)$ for $i = 1$, \ldots, $n$ then, by
Lemmas~\ref{lem:Mint} and~\ref{lem:qsi},
\begin{align*}
X_{\wh{T}} & = %
\sum_{i, j = 1}^n T_i \wh{\pi}( E_{t_i} )_{t_i \vee t_j}
\wh{\pi}( X )_{t_i \vee t_j} %
\wh{\pi}( E_{t_j} )_{t_i \vee t_j} T_j \\[1ex]
 & = \sum_{i, j = 1}^n \Bigl( T_i \wh{\pi}( E_{t_i} )_{t_i \vee t_j} %
\wc{\pi}( X )_{t_i \vee t_j} \wh{\pi}( E_{t_j} )_{t_i \vee t_j} T_j \\
 & \hspace{6em} + \int_{t_i \vee t_j}^\infty I_\mul \otimes T_i %
\wh{\pi}( E_{t_i} )_{t_i \vee t_j} %
\wh{\pi}( \wc{\pi}( X )_{t_i \vee t_j} )_s %
\wh{\pi}( E_{t_j} )_{t_i \vee t_j} T_j \std \Lambda_s \Bigr) \\[1ex]
 & = %
\sum_{i, j = 1}^n \Bigl( T_i E_{t_i} X_{t_i \vee t_j} E_{t_j} T_j \\
 & \hspace{2em} + \int_0^\infty I_\mul \otimes %
1_{[ t_i, \infty )}( s ) 1_{[ t_j, \infty )}( s ) T_i %
\wh{\pi}( E_{t_i} )_s \wh{\pi}( X_{t_i \vee t_j} )_s %
\wh{\pi}( E_{t_j} )_s T_j \std \Lambda_s \Bigr) \\[1ex]
 & = X_{\wc{T}} + \int_0^\infty I_\mul \otimes %
T\bigl( [ 0, s ] \bigr) \wh{\pi}( X_{\wc{T}} )_s %
T\bigl( [ 0, s ] \bigr) \std \Lambda_s;
\end{align*}
the penultimate equality holds because
$\wh{\pi}( E_r )_s \wh{\pi}( E_s )_t = \wh{\pi}( E_r )_t$ if
$r \le s \le t$, and the final equality holds because
\begin{align*}
T\bigl( [ 0, s ] \bigr) & \wh{\pi}( X_{\wc{T}} )_s %
T\bigl( [ 0, s ] \bigr) \\[1ex]
 & = \sum_{i, j, k, l = 1}^n %
1_{[ t_i, \infty )}( s ) 1_{[ t_j, \infty )}( s ) T_i %
\wh{\pi}( T_k E_{t_k} X_{t_k \vee t_l} E_{t_l} T_l )_s T_j \\[1ex]
 & = \sum_{i, j, k, l = 1}^n %
1_{[ t_i, \infty )}( s ) 1_{[ t_j, \infty )}( s ) %
\wh{\pi}( T_i T_k )_s \wh{\pi}( E_{t_k} X_{t_k \vee t_l} E_{t_l} )_s %
\wh{\pi}( T_l T_j )_s \\[1ex]
 & = \sum_{i, j = 1}^n %
1_{[ t_i, \infty )}( s ) 1_{[ t_j, \infty )}( s ) %
T_i \wh{\pi}( E_{t_i} )_s \wh{\pi}( X_{t_i \vee t_j} )_s %
\wh{\pi}( E_{t_j} )_s T_j.
\end{align*}
Similarly, using the second identity in Lemma~\ref{lem:Mint},
\begin{align*}
X_{\wh{T}} - X_{\wc{T}} & = \sum_{i, j = 1}^n %
\int_{t_i \vee t_j}^\infty I_\mul \otimes %
T_i \wh{\pi}( E_{t_i} )_{t_i \vee t_j} %
\wc{\pi}( \wh{\pi}( X )_{t_i \vee t_j} )_s %
\wh{\pi}( E_{t_j} )_{t_i \vee t_j} T_j \std \Lambda_s \\[1ex]
 & = \sum_{i, j = 1}^n \int_0^\infty I_\mul \otimes %
1_{[ t_i \vee t_j, \infty )}( s ) E_s T_i %
\wh{\pi}( E_{t_i} X_{t_i \vee t_j} E_{t_j} )_{t_i \vee t_j} %
T_j E_s \std \Lambda_s \\[1ex]
 & = \int_0^\infty I_\mul \otimes T\bigl( [ 0, s ] \bigr) %
\wc{\pi}( X_{\wh{T}} )_s T\bigl( [ 0, s ] \bigr) \std \Lambda_s,
\end{align*}
since if $s \ge t_i \vee t_j$ then
\begin{align*}
T_i \wc{\pi}( X_{\wh{T}} )_s T_j & = %
\sum_{k, l = 1}^n E_s T_i T_k %
\wh{\pi}( E_{t_k} X_{t_k \vee t_l} E_{t_l} )_{t_k \vee t_l} %
T_l T_j E_s \\[1ex]
 & = E_s T_i \wh{\pi}( E_{t_i} X_{t_i \vee t_j} E_{t_j} ) T_j E_s.
\qedhere
\end{align*}
\end{proof}

\begin{remark}
The integrals on the right-hand sides of (\ref{eqn:idclosedmg}) and
(\ref{eqn:vacclosedmg}) may be viewed as artifacts produced by working
with identity-adapted processes rather than vacuum-adapted ones.
\end{remark}

\section{Stopping martingales at general times}\label{sec:closedmgs}

\begin{definition}
If $M = ( M_t )_{t \in \R_+}$ is a martingale closed by~$M_\infty$ and
$T$ is a discrete quantum stopping time with support
$\{ t_1 < \cdots < t_n \}$ then Lemma~\ref{lem:discmg} implies that
\[
M_{\wc{T}} = E_T M_{t_n} E_T = E_T M_{t_n} E_{t_n} E_T = %
E_T E_{t_n} M_\infty E_{t_n} E_T = E_T M_\infty E_T.
\]
Hence the result of applying vacuum-adapted stopping to a martingale
$M$ closed by $M_\infty$ at a quantum stopping time~$S$ is defined to
be
\[
M_{\wc{S}} := E_S M_\infty E_S.
\]
Note that $E = ( E_t )_{t \in \R_+}$ is a vacuum-adapted martingale
closed by $E_\infty = I$ and $I_{\wc{S}} = E_S$ for any quantum
stopping time $S$. Furthermore, the map~$S \mapsto M_{\wc{S}}$ is
continuous if the collection of quantum stopping times is equipped
with the topology described in Remark~\ref{rem:qsconv} and
$\bop{\fock}{}$ has the strong operator topology.
\end{definition}

\begin{theorem}
Let $M$ be a closed martingale and let $S$ be a quantum stopping time.
\begin{mylist}
\item[\tu{(i)}] $M_{\wc{S}} = E_S M_{\wc{S}} E_S$.
\item[\tu{(ii)}] $\wc{\pi}( M_{\wc{S \wedge t}} )_t = M_{\wc{S \wedge t}}$
for all $t \in [ 0, \infty ]$.
\item[\tu{(iii)}] The process
$M^{\wc{S}} := \bigl( M_{\wc{S \wedge t}} \bigr)_{t \in \R_+}$ is a
vacuum-adapted martingale closed by $M_{\wc{S}}$ and such that
$M^{\wc{S}}_t = E_{S \wedge t} M_{\wc{S}} E_{S \wedge t}$ for all
$t \in \R_+$.
\item[\tu{(iv)}]
$S\bigl( [ 0, t ] \bigr) M_{\wc{S}} S\bigl( [ 0, t ] \bigr) = %
S\bigl( [ 0, t ] \bigr) M_{\wc{S \wedge t}} S\bigl( [ 0, t ] \bigr)$
for all $t \in [ 0, \infty ]$.
\end{mylist}
\end{theorem}
\begin{proof}
The first claim is immediate, because $E_S$ is idempotent.
Furthermore, if~$t \in \R_+$ then
\[
E_t M_{\wc{S \wedge t}} E_t = %
E_t E_{S \wedge t} M_\infty E_{S \wedge t} E_t = %
E_{S \wedge t} M_\infty E_{S \wedge t} = M_{\wc{S \wedge t}},
\]
by Theorems~\ref{thm:exporder}(ii) and~\ref{thm:qstorder};
Theorem~\ref{thm:exporder}(ii) also implies that
\[
E_t M_{\wc{S}} E_t = E_t E_S M_\infty E_S E_t = %
E_{S \wedge t} M_\infty E_{S \wedge t} = M_{\wc{S \wedge t}} = %
M^{\wc{S}}_t,
\]
so $M^{\wc{S}}$ is a vacuum-adapted martingale closed by $M_{\wc{S}}$,
by Proposition~\ref{prp:closedvacmg}. If~$t \in [ 0, \infty ]$ then
\[
E_{S \wedge t} M_{\wc{S}} E_{S \wedge t} = %
E_{S \wedge t} E_S M_\infty E_S E_{S \wedge t} = %
E_{S \wedge t} M_\infty E_{S \wedge t} = M_{\wc{S \wedge t}},
\]
by Theorems~\ref{thm:qstorder} and~\ref{thm:exporder}(i), so (iii)
holds. Finally, Proposition~\ref{prp:Sdecomp} implies that
\[
S\bigl( [ 0, t ] \bigr) E_S M_\infty E_S S\bigl( [ 0, t ] \bigr) = %
S\bigl( [ 0, t ] \bigr) E_{S \wedge t} M_\infty E_{S \wedge t} %
S\bigl( [ 0, t ] \bigr).
\]
\end{proof}

\begin{proposition}
If $M$ is a martingale closed by $M_\infty$ then
$\expn_\Vac\bigl[ M_{\wc{T}} \bigr] = \expn_\Vac[ M_0 ]$ for any
quantum stopping time $T$, where $\expn_\Vac$ is the vacuum state
\tu{(\ref{eqn:vacuumstate})}.
\end{proposition}
\begin{proof}
Note that
\[
\expn_\Vac\bigl[ M_{\wc{T}} \bigr] = %
\langle \evec{0}, E_T M_\infty E_T \evec{0} \rangle = %
\langle \evec{0}, E_0 E_T M_\infty E_T E_0 \evec{0} \rangle = %
\langle \evec{0}, M_0 \evec{0} \rangle = \expn_\Vac[ M_0 ].
\qedhere
\]
\end{proof}

\begin{theorem}[Optional Sampling]\label{thm:vacoptional}
Let $M$ be a closed martingale and let~$S$ and $T$ be quantum stopping
times with $S \le T$. Then
\[
( M_{\wc{S}} )_{\wc{T}} = M_{\wc{S}} = ( M_{\wc{T}} )_{\wc{S}}.
\]
\end{theorem}
\begin{proof}
Suppose $M$ is closed by $M_\infty$. Theorem~\ref{thm:exporder}(i)
implies that
\[
( M_{\wc{T}} )_{\wc{S}} = E_S E_T M_\infty E_T E_S = %
E_S M_\infty E_S = M_{\wc{S}}
\]
and
\[
( M_{\wc{S}} )_{\wc{T}} = E_T E_S M_\infty E_S E_T = %
E_S M_\infty E_S = M_{\wc{S}}.
\qedhere
\]
\end{proof}

Theorem~\ref{thm:idclosedmg} motivates the definition in the following
theorem, which was established by Coquio \cite[Theorem~3.5]{Coq06};
the proof given here is a shortening of hers.

\begin{theorem}\label{thm:idstopop}
Let $Z \in \bop{\fock}{}$ and let $S$ be a quantum stopping time. Then
\begin{equation}\label{eqn:idstopop}
Z_{\wh{S}} := Z_{\wc{S}} + \int_0^\infty I_\mul \otimes %
S\bigl( [ 0, s ] \bigr) \wh{\pi}\bigl( Z_{\wc{S}} \bigr)_s %
S\bigl( [ 0, s ] \bigr) \std \Lambda_s
\end{equation}
extends to an element of $\bop{\fock}{}$, denoted the same way,
with $\| Z_{\wh{S}} \| \le \| Z \|$.
\end{theorem}
\begin{proof}
Suppose first that $S$ is discrete and let $Y$ denote the integral on
the right-hand side of~(\ref{eqn:idstopop}). From
Proposition~\ref{prp:procSad} and the remark at the end of
Definition~\ref{def:discstop}, it follows that
\[
Z_{\wc{S}} = E_S Z_{\wh{S}} = E_S Z_{\wc{S}} + E_S Y,
\]
so $E_S Y = 0$ and
\[
\wh{\pi}( E_S )_t \wh{\pi}( Y )_t = %
\wh{\pi}( E_S )_t \int_0^t I_\mul \otimes %
S\bigl( [ 0, s ] \bigr) \wh{\pi}\bigl( Z_{\wc{S}} \bigr)_s %
S\bigl( [ 0, s ] \bigr) \std \Lambda_s = 0
\]
for all $t \in [ 0, \infty ]$. Thus if $\theta \in \evecs$ then, as
$S\bigl( [ 0, s ] \bigr)$ commutes with $\wh{\pi}( E_S )_s$ for all
$s \in [ 0, \infty ]$, the weak form of the quantum It\^o product
formula for gauge integrals, Theorem~\ref{thm:weakqipf}, implies that
\begin{align*}
\| Z_{\wh{S}} \theta \|^2 - \| Z_{\wc{S}} \theta \|^2 & = %
\| ( Z_{\wh{S}} - Z_{\wc{S}} ) \theta \|^2 \\[1ex]
 & = \int_0^\infty \| ( I_\mul \otimes S\bigl( [ 0, s ] \bigr) %
\wh{\pi}\bigl( Z_{\wc{S}} \bigr)_s %
S\bigl( [ 0, s ] \bigr) ) \nabla_s \theta \|^2 \std s \\[1ex]
 & \le \| Z \|^2 \int_0^\infty \| ( I_\mul \otimes %
S\bigl( [ 0, s ] \bigr) \wh{\pi}\bigl( E_S \bigr)_s %
S\bigl( [ 0, s ] \bigr) ) \nabla_s \theta \|^2 \std s \\[1ex]
 & = \| Z \|^2 \bigl( \| \theta \|^2 - \| E_S \theta \|^2 \bigr),
\end{align*}
where
\[
\nabla : \evecs \to L^2( \R_+; \mul \otimes \fock ); \ %
\big( \nabla \evec{f} \bigr)( t ) = \nabla_t \evec{f} := %
f( t ) \evec{f}
\]
is the linear gradient operator; the final identity follows from the
first line by taking $Z = I$. Hence
\[
\| Z_{\wh{S}} \theta \|^2 %
\le \| Z \|^2 \bigl( \| E_S \theta \|^2 + \| \theta \|^2 - %
\| E_S \theta \|^2 ) = \| Z \|^2 \| \theta \|^2,
\]
and $Z_{\wh{S}}$ extends as claimed.

For a general quantum stopping time $S$, let $S_n$ be a sequence of
discrete quantum stopping times such that $S_n \qsto S$. Then
$Z_{\wc{S_n}} \to Z_{\wc{S}}$ in the strong operator topology and
\begin{multline*}
\int_0^\infty I_\mul \otimes %
S_n\bigl( [ 0, s ] \bigr) \wh{\pi}\bigl( Z_{\wc{S_n}} \bigr)_s %
S_n\bigl( [ 0, s ] \bigr) \std \Lambda_s \\[1ex]
\to \int_0^\infty I_\mul \otimes %
S\bigl( [ 0, s ] \bigr) \wh{\pi}\bigl( Z_{\wc{S}} \bigr)_s %
S\bigl( [ 0, s ] \bigr) \std \Lambda_s
\end{multline*}
in the strong operator topology on $\evecs$, by
Lemma~\ref{lem:qsiestimate}. Hence
$\| Z_{\wc{S}} \theta \| \le \| Z \| \, \| \theta \|$ for all
$\theta \in \evecs$ and the result follows.
\end{proof}

\begin{remark}
It is readily verified that $( Z, \, S ) \mapsto Z_{\wh{S}}$ is
jointly continuous on the product of any bounded subset of
$\bop{\fock}{}$ with the collection of all quantum stopping times,
when $\bop{\fock}{}$ is equipped with the strong operator topology and
the collection of all quantum stopping times is given the topology of
Remark~\ref{rem:qsconv}.

Consequently,
\[
Z_{\wh{S}} = %
Z_{\wc{S}} + \int_0^\infty I_\mul \otimes %
S\bigl( [ 0, s ] \bigr) \wc{\pi}( Z_{\wh{S}} )_s %
S\bigl( [ 0, s ] \bigr) \std \Lambda_s
\]
for any $Z \in \bop{\fock}{}$ and any quantum stopping time $S$, since
this identity agrees with (\ref{eqn:vacclosedmg}) when~$S$ is discrete
and extends to the general case by approximation.

Furthermore, as $I_{\wh{T}} = I$ for any discrete quantum stopping
time $T$, so
\begin{multline}
I = I_{\wh{S}} = E_S + \int_0^\infty I_\mul \otimes %
S\bigl( [ 0, s ] \bigr) \wh{\pi}( E_S )_s \std \Lambda_s \\[-2ex]
\iff E_S = I - \int_0^\infty I_\mul \otimes %
S\bigl( [ 0, s ] \bigr) \wh{\pi}( E_S )_s \std \Lambda_s
\label{eqn:tpiint}
\end{multline}
for any quantum stopping time $S$. This identity, which is believed to
be novel, expresses the time projection~$E_S$ using an
identity-adapted gauge integral; it should be compared with the first
identity in (\ref{eqn:tpint}).
\end{remark}

The following result \cite[Proposition~3.11]{Coq06} is the
identity-adapted counterpart of Theorem~\ref{thm:qstint}.

\begin{proposition}\label{prp:idtpint}
Let $S$ be a quantum stopping time. If $t \in [ 0, \infty ]$ then
\[
\wh{\pi}( E_S )_t = I - \int_0^t %
I_\mul \otimes S\bigl( [ 0, s ] \bigr) \wh{\pi}( E_S )_s %
\std \Lambda_s = %
E_S + \int_t^\infty %
I_\mul \otimes S\bigl( [ 0, s ] \bigr) \wh{\pi}( E_S )_s %
\std \Lambda_s.
\]
\end{proposition}
\begin{proof}
The first identity follows from (\ref{eqn:tpiint}) and the fact that
the gauge integral of an identity-adapted process is an
identity-adapted martingale. Alternatively, applying
Theorem~\ref{thm:switch} to~(\ref{eqn:tpint}) gives that
\[
\wh{\pi}( E_S )_t - I = \wh{\pi}( E_S - I )_t = %
-\int_0^t I_\mul \otimes ( S\bigl( [ 0, s ] \bigr) + %
\wh{\pi}( E_S - I )_s ) \std \Lambda_s,
\]
so
\[
\wh{\pi}( E_S )_t = I - \int_0^t I_\mul \otimes \bigl( %
S\bigl( [ 0, s ] \bigr) \wh{\pi}( E_S )_s + %
S\bigl( ( s, \infty ] \bigr) %
( I - \wh{\pi}( E_S )_s ) \bigr) \std \Lambda_s.
\]
This gives the first claim, since Theorem~\ref{thm:exporder}(iv) and
Proposition~\ref{prp:Sdecomp} imply that
\begin{equation}\label{eqn:Safters}
S\bigl( ( s, \infty ] \bigr) \wh{\pi}( E_S )_s = %
S\bigl( ( s, \infty ] \bigr) \wh{\pi}( E_{S \wedge s} )_s = %
S\bigl( ( s, \infty ] \bigr) \wh{\pi}( E_s )_s = %
S\bigl( ( s, \infty ] \bigr).
\end{equation}
For the second, note that the first claim with $t = \infty$ yields the
identity
\[
E_S + \int_t^\infty I_\mul \otimes %
S\bigl( [ 0, s ] \bigr) \wh{\pi}( E_S )_s \std \Lambda_s = %
I - \int_0^t I_\mul \otimes %
S\bigl( [ 0, s ] \bigr) \wh{\pi}( E_S )_s \std \Lambda_s.
\qedhere
\]
\end{proof}

\begin{proposition}\label{prp:stopidem}
If $Z \in \bop{\fock}{}$ and $S$ is a quantum stopping time then
\[
( Z_{\wc{S}} )_{\wc{S}} = Z_{\wc{S}}, \qquad %
( Z_{\wh{S}} )_{\wc{S}} = Z_{\wc{S}}, \qquad %
( Z_{\wc{S}} )_{\wh{S}} = Z_{\wh{S}} \qquad \text{and} \qquad %
( Z_{\wh{S}} )_{\wh{S}} = Z_{\wh{S}}.
\]
Furthermore, $E_S Z_{\wh{S}} = Z_{\wc{S}}$.
\end{proposition}
\begin{proof}
The first identity is trivial, and the third follows immediately from
it. The fifth is true if~$S$ is discrete, as noted in the proof of
Theorem~\ref{thm:idstopop}, so holds in general by approximation. The
second is now immediate, and the fourth identity follows from the
third and the second.
\end{proof}

\begin{remark}
If $Z$, $W \in \bop{\fock}{}$ and $S$ is a quantum stopping time then
working as in the proof of Theorem~\ref{thm:idstopop} gives that
\[
( Z_{\wh{S}} W )_{\wh{S}} - Z_{\wh{S}} W_{\wh{S}} = %
\int_0^\infty I_\mul \otimes S\bigl( [ 0, s ] \bigr) %
\wh{\pi}( Z_{\wc{S}} )_s S\bigl( ( s, \infty ] \bigr) %
\wh{\pi}( W_{\wc{S}} )_s S\bigl( [ 0, s ] \bigr) \std \Lambda_s.
\]
(\Cf the formula given in \cite[Remark~3.13]{Coq06}.) Thus, as noted
by Coquio, the map~$Z \mapsto Z_{\wh{S}}$ is not, in general, a
conditional expectation on~$\bop{\fock}{}$, in contrast to
vacuum-adapted stopping: see Proposition~\ref{prp:vaccondexp}.
\end{remark}

\begin{definition}[{\cite[Theorem~3.5]{Coq06}}]
The result of identity-adapted stopping a martingale $M$ closed by
$M_\infty$ at a quantum stopping time $S$ is
\[
M_{\wh{S}} := ( M_\infty )_{\wh{S}} = M_{\wc{S}} + %
\int_0^\infty I_\mul \otimes S\bigl( [ 0, s ] \bigr) %
\wh{\pi}\bigl( M_{\wc{S}} \bigr)_s S\bigl( [ 0, s ] \bigr) %
\std \Lambda_s.
\]
\end{definition}

\begin{theorem}\label{thm:admg}
If $S$ is a quantum stopping time and $M$ is a martingale then
\[
\wc{\pi}( M_{\wc{S}} )_t = M_{\wc{S \wedge t}} \qquad \text{and} \qquad %
\wh{\pi}( M_{\wh{S}} )_t = M_{\wh{S \wedge t}} %
\qquad \text{for all } t \in [ 0, \infty ];
\]
in particular, the processes $M^{\wc{S}}$ and $M^{\wh{S}}$ are
martingales closed by $M_{\wc{S}}$.
\end{theorem}
\begin{proof}
The first identity is an immediate consequence of
Theorem~\ref{thm:exporder}(iv), that $E_t E_S = E_{S \wedge t}$. The
second was established by Coquio when $S$ is discrete
\cite[Properties~3.3(3)]{Coq06} and holds in general by an
approximation argument. The final remark follows from
Propositions~\ref{prp:closedvacmg} and~\ref{prp:closedidmg}.
\end{proof}

The following result is due to Coquio \cite[Proposition~3.9]{Coq06}.

\begin{theorem}[Optional Sampling]\label{thm:idoptional}
Let $M$ be a closed martingale and let $S$ and $T$ be quantum stopping
times with $S \le T$. Then
\[
( M_{\wh{T}} )_{\wh{S}} = M_{\wh{S}}.
\]
\end{theorem}
\begin{proof}
Note first that, since $E_S E_T = E_S$, so
\[
( M_{\wh{T}} )_{\wc{S}} = E_S M_{\wh{T}} E_S = %
E_S E_T M_{\wh{T}} E_S = E_S M_{\wc{T}} E_S = M_{\wc{S}},
\]
where the penultimate equality follows from
Proposition~\ref{prp:stopidem}. Hence
\begin{align*}
( M_{\wh{T}} )_{\wh{S}} & = ( M_{\wh{T}} )_{\wc{S}} + %
\int_0^\infty I_\mul \otimes S\bigl( [ 0, s ] \bigr) %
\wh{\pi}\bigl( ( M_{\wh{T}} )_{\wc{S}} \bigr)_s %
S\bigl( [ 0, s ] \bigr) \std \Lambda_s \\[1ex]
 & = M_{\wc{S}} + \int_0^\infty I_\mul \otimes %
S\bigl( [ 0, s ] \bigr) \wh{\pi}\bigl( M_{\wc{S}} \bigr)_s %
S\bigl( [ 0, s ] \bigr) \std \Lambda_s \\[1ex]
 & = M_{\wh{S}}.
\qedhere
\end{align*}
\end{proof}

\begin{question}
Suppose $S$ and $T$ are quantum stopping times, with $S \le T$.
If~$Z \in \bop{\fock}{}$ then
\begin{equation}\label{eqn:same}
( Z_{\wc{S}} )_{\wc{T}} = ( Z_{\wc{T}} )_{\wc{S}} = Z_{\wc{S}} = %
( Z_{\wh{T}} )_{\wc{S}},
\end{equation}
by Theorem~\ref{thm:exporder} and the proof of
Theorem~\ref{thm:idoptional}, whereas
\begin{equation}
( Z_{\wc{T}} )_{\wh{S}} = Z_{\wh{S}} = ( Z_{\wh{T}} )_{\wh{S}},
\end{equation}
by the second identity in (\ref{eqn:same}) and
Theorem~\ref{thm:idoptional}. What are
\[
( Z_{\wh{S}} )_{\wc{T}} , \qquad ( Z_{\wc{S}} )_{\wh{T}} \qquad %
\text{and} \qquad ( Z_{\wh{S}} )_{\wh{T}}?
\]
\end{question}

\section{Stopping closed FV processes}\label{sec:FV}

\begin{definition}
The process $Y = ( Y_t )_{t \in \R_+}$ is said to be an
\emph{FV process} if there exists an \emph{integrand process}~$H$ such
that $s \mapsto H_s x$ is strongly measurable for all $x \in \fock$
and $\| H \| : s \mapsto \| H_s \|$ is locally integrable, so that the
integral
\[
Y_t = \int_0^t H_s \std s
\]
exists pointwise as a Bochner integral for all $t \in \R_+$; we
write~$Y = \int H \std t$ to denote this. If $\| H \|$ is integrable
then $Y_\infty := \int_0^\infty H_t \std t$ exists pointwise and the
FV process $Y$ is \emph{closed} by~$Y_\infty$.

The sets of FV processes and closed FV processes are subalgebras of
the algebra of processes.

If the integrand process $H$ is identity adapted or vacuum adapted
then $Y$ has the same property.
\end{definition}

\begin{definition}
If $Y = \int H \std t$ is an FV process which is closed by $Y_\infty$
and the discrete quantum stopping time  $T$ has support
$\{ t_1 < \cdots < t_n \}$ then, letting
$T_i := T\bigl( \{ t_i \} \bigr)$ as usual,
\begin{align*}
Y_{\wc{T}} & = \sum_{i, j = 1}^n %
T_i E_{t_i} Y_{t_i \vee t_j} E_{t_j} T_j \\[1ex]
 & = E_T Y_\infty E_T - \sum_{i, j = 1}^n \int_{t_i \vee t_j}^\infty %
T_i E_{t_i} H_s E_{t_j} T_j \std s \\[1ex]
 & = E_T Y_\infty E_T - %
\int_0^\infty \sum_{i, j = 1}^n 1_{[ t_i, \infty )}( s ) T_i E_{t_i} %
H_s 1_{[ t_j, \infty )}( s ) E_{t_j} T_j \std s \\[1ex]
 & = E_T Y_\infty E_T - E_T \int_0^\infty T\bigl( [ 0, s ] \bigr) %
H_s T\bigl( [ 0, s ] \bigr) \std s \, E_T.
\end{align*}
Hence, for any quantum stopping time $S$ we let
\begin{align*}
Y_{\wc{S}} & := E_S \Bigl( Y_\infty - \int_0^\infty %
S\bigl( [ 0, s ] \bigr) H_s %
S\bigl( [ 0, s ] \bigr) \std s \Bigr) E_S \\[1ex]
 & \hphantom{:}= \int_0^\infty E_S \Bigl( H_s - %
S\bigl( [ 0, s ] \bigr) H_s %
S\bigl( [ 0, s ] \bigr) \Bigr) E_S \std s.
\end{align*}
Note that $E_S Y_{\wc{S}} = Y_{\wc{S}} = Y_{\wc{S}} E_S$, and that
$S \mapsto Y_{\wc{S}}$ is continuous when the collection of quantum
stopping times has the topology described in Remark~\ref{rem:qsconv}
and $\bop{\fock}{}$ is equipped with the strong operator topology.
\end{definition}

\begin{remark}
If $S$ corresponds to a classical stopping time $\tau$ and
$Y = \int H \std t$ is a classical FV process, for any $\omega$ in the
underlying sample space,
\begin{align*}
\int_0^\infty \Bigl( H_s - %
S\bigl( [ 0, s ] \bigr) H_s S\bigl( [ 0, s ] \bigr) \Bigr)%
( \omega ) \std s & = %
\int_0^\infty \bigl( S\bigl( ( s, \infty ] \bigr)%
H_s \bigr)( \omega ) \std s \\[1ex]
 & = \int_0^{\tau( \omega )} H_s( \omega ) \std s \\[1ex]
& = ( Y_\tau ) ( \omega ).
\end{align*}
\end{remark}

\begin{proposition}
If $S$ is a quantum stopping time, $Y$ is an FV process which is
closed by $Y_\infty$ and $t \in [ 0, \infty ]$ then
\[
S\bigl( [ 0, t ] \bigr) Y_{\wc{S}} S\bigl( [ 0, t ] \bigr) = %
S\bigl( [ 0, t ] \bigr) Y_{\wc{S \wedge t}} S\bigl( [ 0, t ] \bigr).
\]
\end{proposition}
\begin{proof}
Note that, by the definition of~$Y_{\wc{S}}$ and
Proposition~\ref{prp:Sdecomp},
\begin{align*}
S\bigl( [ 0, t ] \bigr) Y_{\wc{S}} S\bigl( [ 0, t ] \bigr) & = %
S\bigl( [ 0, t ] \bigr) E_{S \wedge t} Y_\infty E_{S \wedge t} %
S\bigl( [ 0, t ] \bigr) \\
 & \qquad - E_{S \wedge t} \int_0^\infty %
S\bigl( [ 0, s \wedge t ] \bigr) H_s %
S\bigl( [ 0, s \wedge t ] \bigr) \std s \, E_{S \wedge t} \\[1ex]
 & = S\bigl( [ 0, t ] \bigr) Y_{S \wedge t} S\bigl( [ 0, t ] \bigr),
\end{align*}
since
$S\bigl( [ 0, s \wedge t ] \bigr) = %
S\bigl( [ 0, t ] \bigr) ( S \wedge t )\bigl( [ 0, s ] \bigr)$
for all $s \in \R_+$.
\end{proof}

\begin{definition}
A \emph{semimartingale} is a process of the form $X = M + Y$,
where~$M$ is a martingale closed by $M_\infty$ and $Y = \int H \std t$
is a FV process closed by $Y_\infty$, with the integrand process~$H$
is identity or vacuum adapted.

(Strictly speaking, this is a \emph{closed} semimartingale, but no
other sort of semimartingale will be considered.)

For any quantum stopping time~$S$, the stopped semimartingale
\begin{align}\label{eqn:ssmg}
X_{\wc{S}} & := M_{\wc{S}} + Y_{\wc{S}} \nonumber \\[1ex]
 & \hphantom{:}= E_S ( M_\infty + Y_\infty ) E_S - %
\int_0^\infty E_S S\bigl( [ 0, s ] \bigr) H_s %
S\bigl( [ 0, s ] \bigr) E_S \std s.
\end{align}
\end{definition}

The following lemma shows that the decomposition $X = M + Y$ is
unique; it follows that~(\ref{eqn:ssmg}) is a good definition.

\begin{lemma}\label{lem:decomp}
Let $M$ be a martingale, let $Y = \int H \std t$ be an FV process and
suppose the integrand process $H$ is either identity adapted or vacuum
adapted. If~$M + Y = 0$ then $M = 0$ and~$Y = 0$.
\end{lemma}
\begin{proof}
It suffices to prove that $Y = 0$. To see this, first recall that $M$
is a martingale if and only if
$E_s M_t E_s = M_s E_s = E_s M_s$ for all $s$, $t \in \R_+$ such
that~$s \le t$. Then, for such $s$ and $t$,
\[
0 = E_s ( M_t + Y_t - M_s - Y_s ) E_s = E_s ( Y_t - Y_s ) E_s = %
\int_s^t E_s H_r E_s \std r;
\]
given $x \in \fock$ it follows that
$\wc{\pi}( H )_s x = E_s H_s E_s x = 0$ for almost all $s \in \R_+$,
so~$s \mapsto H_s x = 0$ almost everywhere and the claim follows.
\end{proof}

The next result was observed by Coquio
\cite[Proof of Proposition~3.15]{Coq06}.

\begin{lemma}\label{lem:idfvint}
Let the FV process $Y = \int H \std t$ be closed by $Y_\infty$.
Then
\[
\wh{\pi}( Y )_t = E_t Y_\infty E_t + %
\int_t^\infty I_\mul \otimes \wh{\pi}( E_t Y_s E_t )_s %
\std \Lambda_s - %
\int_t^\infty \wh{\pi}( E_t H_s E_t )_s \std s
\]
for all $t \in [ 0, \infty ]$.
\end{lemma}
\begin{proof}
If $f$, $g \in \elltwo$ then 
\begin{align*}
\langle \evec{f}&, \int_t^\infty I_\mul \otimes %
\wh{\pi}( E_t Y_s E_t )_s \std \Lambda_s \, \evec{g} \rangle \\[1ex]
 & = \int_t^\infty \langle f( s ), g( s ) \rangle \, %
\langle \evec{1_{[ 0, t )} f}, Y_s \evec{1_{[ 0, t )} g} \rangle \, %
\langle \evec{1_{[ s, \infty )} f}, %
\evec{1_{[ s, \infty )} g} \rangle \std s \\[1ex]
 & = -\Bigl[ \langle \evec{1_{[ 0, t )} f}, %
Y_s \evec{1_{[ 0, t )} g} \rangle \, %
\langle \evec{1_{[ s, \infty )} f}, %
\evec{1_{[ s, \infty )} g} \rangle \Bigr]_t^\infty \\
 & \qquad + \int_t^\infty \langle \evec{1_{[ 0, t )} f}, %
H_s \evec{1_{[ 0, t )} g} \rangle \, %
\langle \evec{1_{[ s, \infty )} f}, %
\evec{1_{[ s, \infty )} g} \rangle \std s \\[1ex]
 & = \langle \evec{f}, ( \wh{\pi}( Y )_t - E_t Y_\infty E_t ) %
\evec{g} \rangle + \langle \evec{f}, %
\int_t^\infty \wh{\pi}( E_t H_s E_t )_s \std s \, \evec{g} \rangle.
\qedhere
\end{align*}
\end{proof}

\begin{definition}
It follows from Definition~\ref{def:discrete_stop} and
Lemma~\ref{lem:idfvint} that if $T$ is a discrete quantum stopping
time and the FV process $Y = \int H \std t$ is closed by~$Y_\infty$
then
\begin{multline}\label{eqn:idstopfv}
Y_{\wh{T}} = E_T Y_\infty E_T + \int_0^\infty I_\mul \otimes %
T\bigl( [ 0, s ] \bigr) \wh{\pi}( E_T Y_s E_T )_s %
T\bigl( [ 0, s ] \bigr) \std \Lambda_s \\
 - \int_0^\infty T\bigl( [ 0, s ] \bigr) \wh{\pi}( E_T H_s E_T )_s %
T\bigl( [ 0, s ] \bigr) \std s.
\end{multline}
The identity (\ref{eqn:idstopfv}) is used by Coquio
\cite[Proposition~3.15]{Coq06} to motivate the following definition:
if~$S$ is a quantum stopping time and $X = M + Y$ is a semimartingale
then
\begin{align*}
X_{\wh{S}} & := M_{\wh{S}} + Y_{\wh{S}} \\[1ex]
 & \phantom{:}= E_S( M_\infty + Y_\infty ) E_S + %
\int_0^\infty I_\mul \otimes %
S\bigl( [ 0, s ] \bigr) \wh{\pi}( E_S X_s E_S )_s %
S\bigl( [ 0, s ] \bigr) \std \Lambda_s \\
& \hspace{13em} - %
\int_0^\infty S\bigl( [ 0, s ] \bigr) \wh{\pi}( E_S H_s E_S )_s %
S\bigl( [ 0, s ] \bigr) \std s.
\end{align*}
(Recall that $\wh{\pi}( M )_s = \wh{\pi}( M_\infty )_s$ for all
$s \in \R_+$, by the martingale property.)

Note that $S \mapsto Y_S$ is continuous when $\bop{\fock}{}$ is
equipped with the strong operator topology and the set of quantum
stopping times has the topology defined in Remark~\ref{rem:qsconv}.
\end{definition}

\begin{remark}
Let $X = M + Y$ be a semimartingale. Proposition~\ref{prp:procSad}
gives that~$E_S X_{\wh{S}} = X_{\wc{S}}$ when $S$ is discrete, and
therefore this identity holds in general by approximation. As
$E_S S\bigl( [ 0, s ] \bigr) = S\bigl( [ 0, s ] \bigr) E_S = %
S\bigl( [ 0, s ] \bigr) E_S E_s$
for all $s \in [ 0, \infty ]$, by Proposition~\ref{prp:Sdecomp},
so
\[
E_S \int_0^\infty S\bigl( [ 0, s ] \bigr) %
\wh{\pi}( E_S H_s E_S )_s S\bigl( [ 0, s ] \bigr) \std s = %
\int_0^\infty S\bigl( [ 0, s ] \bigr) %
E_S H_s E_S S\bigl( [ 0, s ] \bigr) \std s.
\]
Comparing (\ref{eqn:ssmg}) with $E_S X_{\wh{S}}$, it follows that
\begin{equation}\label{eqn:killedgauge}
E_S \int_0^\infty I_\mul \otimes %
S\bigl( [ 0, s ] \bigr) \wh{\pi}( E_S X_s E_S )_s %
S\bigl( [ 0, s ] \bigr) \std \Lambda_s = 0,
\end{equation}
and the gauge integral can again be seen as an artifact produced by
identity adaptedness.
\end{remark}

\section{Stopping regular semimartingales}\label{sec:vqsm}

For the terminology used throughout this section, see the appendix,
Section~\ref{sec:qsc}.

\begin{example}\label{eg:weyl}
Recall that, for any $f \in \elltwo$, the Weyl operator $W( f )$ is
the unitary operator on $\fock$ such that
\[
W( f ) \evec{g} = %
\exp\bigl( -\hlf \| f \|^2 - \langle f, g \rangle \bigr) %
\evec{g + f} \qquad \text{for all } g \in \elltwo.
\]
Setting
\[
W( f )_t := W( 1_{[ 0, t )} f ) = \wh{\pi}\big( W( f ) \bigr)_t
\quad \text{and} \quad %
V( f )_t := E_t W( 1_{[ 0, t )} f ) = \wc{\pi}\bigl( W( f ) \bigr)_t
\]
produces identity-adapted and vacuum-adapted regular semimartingales,
with
\begin{align*}
W( f )_t & = -\int_0^t \bra{f( s )} \otimes W( f )_s \std A_s + %
\int_0^t \ket{f( s )} \otimes W( f )_s \std A_s \\
 & \hspace{16em} - \frac{1}{2}%
\int_0^t \| f( s ) \|^2 W( f )_s \std s \\[1ex]
\text{and} \quad %
V( f )_t & = \int_0^t I_\mul \otimes V( f )_s \std \Lambda_s - %
\int_0^t \bra{f( s )} \otimes V( f )_s \std A_s \\
 & \hspace{6em} + %
\int_0^t \ket{f( s )} \otimes V( f )_s \std A_s - %
\frac{1}{2} \int_0^t \| f( s ) \|^2 V( f )_s \std s
\end{align*}
for all $t \in [ 0, \infty ]$; the operator
$\bra{x} \in \bop{\mul}{\C}$ is such that
$\bra{x} y = \langle x, y \rangle$ for all~$x$, $y \in \mul$, and
its adjoint $\ket{x} \in \bop{\C}{\mul}$ is such
that~$\ket{x} \lambda = \lambda x$ for all $\lambda \in \C$
and~$x \in \mul$. (\Cf Theorem~\ref{thm:switch}.)
\end{example}

The class of regular $\Omega$-semimartingales is closed under
vacuum-adapted stopping; the following result gives this in explicit
form. Coquio obtained an analogous result for regular quantum
semimartingales \cite[Proposition~3.16]{Coq06}.

\begin{theorem}\label{thm:vacsmg}
Let the regular $\Omega$-semimartingale $X = M + Y$ have martingale
part $M = \int N \std \Lambda + P \std A + Q \std A^\dagger$ and FV
part $Y = \int R \std t$. If $S$ is a quantum stopping
time then
\[
X_{\wc{S}} = %
\int_0^\infty {}^{\wc{S}} N_s \std \Lambda_s + %
\int_0^\infty {}^{\wc{S}} P_s \std A_s + %
\int_0^\infty {}^{\wc{S}} Q_s \std A^\dagger_s + %
\int_0^\infty {}^{\wc{S}} R_s \std s,
\]
where, for all $s \in \R_+$,
\begin{align*}
{}^{\wc{S}} N_s & := %
( I_\mul \otimes S\bigl( ( s, \infty ] \bigr) ) N_s %
( I_\mul \otimes S\bigl( ( s, \infty ] \bigr) ), \\[1ex]
{}^{\wc{S}} P_s & := %
E_S P_s ( I_\mul \otimes S\bigl( ( s, \infty ] \bigr) ) \\[1ex]
{}^{\wc{S}} Q_s & := %
( I_\mul \otimes S\bigl( ( s, \infty ] \bigr) ) Q_s E_S \\[1ex]
\text{and} \quad {}^{\wc{S}} R_s & := E_S R_s E_S - %
S\bigl( [ 0, s ] \bigr) E_S R_s E_S S\bigl( [ 0, s ] \bigr).
\end{align*}
\end{theorem}
\begin{proof}
If $X = \int N \std \Lambda + P \std A + Q \std A^\dagger + R \std t$
and $X' = \int N' \std \Lambda$ are regular $\Omega$-semimartingales
then Theorem~\ref{thm:qipf} implies that
\begin{align*}
X X' & = \int N N' \std \Lambda + P N' \std A + %
Q X' \std A^\dagger + R X' \std t \\[1ex]
X' X & = \int N' N \std \Lambda + X' P \std A + %
N' Q \std A^\dagger + X' R \std t \\[1ex]
\text{and} \quad X' X X' & = %
\int N' N N' \std \Lambda + X' P N' \std A + %
N' Q X' \std A^\dagger + X' R X' \std t,
\end{align*}
so
\begin{align*}
X' & X X' - X' X - X X' + X \\[1ex] 
 & = \int ( N' N N' - N' N - N N' + N ) \std \Lambda + %
( X' P N' - X' P - P N' + P ) \std A \\
 & \qquad + ( N' Q X' - N' Q - Q X' + Q ) \std A^\dagger + %
( X' R X' - X' R - R X' + R ) \std t \\[1ex]
 & = \int ( I_{\mul \otimes \fock} - N' ) N %
( I_{\mul \otimes \fock} - N' ) \std \Lambda + %
( I - X' ) P ( I_{\mul \otimes \fock} - N' ) \std A \\
 & \hspace{7em} + %
( I_{\mul \otimes \fock} - N' ) Q ( I - X' ) \std A^\dagger + %
( I - X' ) R ( I - X' ) \std t.
\end{align*}
Taking
\[
X'_t = I - E_{S \wedge t} = %
\int_0^t I_\mul \otimes S\bigl( [ 0, s ] \bigr) E_s \std \Lambda_s %
\qquad \text{for all } t \in [ 0, \infty ]
\]
now gives the result, since
\begin{align*}
& E_S X_\infty E_S \\[1ex]
 & = ( X' X X' - X' X - X X' + X )_\infty \\[1ex]
 & = \int_0^\infty ( I_\mul \otimes S\bigl( ( s, \infty ] \bigr) ) %
N_s ( I_\mul \otimes S\bigl( ( s, \infty ] \bigr) ) \std \Lambda_s + %
E_S P_s ( I_\mul \otimes S\bigl( ( s, \infty ] \bigr) ) \std A_s \\
 & \hspace{12em} + ( I_\mul \otimes S\bigl( ( s, \infty ] \bigr) ) %
Q_s E_S \std A^\dagger_s + E_S R_s E_S \std s.
\qedhere
\end{align*}
\end{proof}

\section{Stopping the future projection}\label{sec:postS}

\begin{definition}
Given $t \in \R_+$, let the isometric right shift
\[
\theta_t : \elltwo \to \elltwo; \ %
( \theta_t f )( s ) := 1_{[ t, \infty )}( s ) f( s - t ),
\]
and let $\Gamma_t \in \bop{\fock}{}$ be its second
quantisation, so that
$\Gamma_t^* \evec{f} = \evecc\bigl( f( \cdot + t ) \bigr)$ for
all~$f \in \elltwo$. If $F_t := \Gamma_t \Gamma_t^*$ then
\[
F_t \evec{f} = \evec{1_{[ t, \infty )} f} %
\qquad \text{for all } f \in \elltwo
\]
and $F$ is an identity-adapted martingale closed by $F_\infty := E_0$.

Given $s$, $t \in [ 0, \infty ]$ with $s \le t$, note that
\[
E_s \Gamma_t = E_0 = \Gamma_t E_s \quad \text{and} \quad
E_s F_t = E_0 = F_t E_s.
\]
(The operator $F_t$ is denoted by $R_t$ in~\cite[Section~4]{Coq06} and
by~$\Gamma( \chi_{[ t, \infty )} )$ in~\cite{PaS87}.)
\end{definition}

\begin{remark}
If the discrete quantum stopping time $T$ has
$\{ t_1 < \cdots < t_n \}$ and~$T_i := T\bigl( \{ t_i \} \bigr)$
for $i = 1, \ldots, n$ then
\[
\Gamma_{\wh{T}} = \sum_{i, j}^n T_i \wh{\pi}( E_{t_i} %
\Gamma_{t_i \vee t_j} E_{t_j} )_{t_i \vee t_j} T_j = %
\sum_{i, j = 1}^n T_i \wh{\pi}( E_0 )_{t_i \vee t_j} T_j = %
F_T F_T^*,
\]
where $F_T := \int_{[ 0, \infty ]} T( \rd s ) \, F_s$; note that
$F_{s \vee t} = F_s F_t$ for all $s$, $t \in [ 0, \infty ]$.

As $\Gamma$ is not a semimartingale, it is unclear how to extend this
identity to hold for more general stopping times; Parthasarathy and
Sinha introduced a left-stopped form of $\Gamma$ in
\cite[Section~5]{PaS87}; see also \cite[Theorem~3.8]{BeS14}.

However, similar working shows that
\[
F_{\wh{T}} = \sum_{i, j = 1}^n T_i \wh{\pi}( E_{t_i} %
F_{t_i \vee t_j} E_{t_j} )_{t_i \vee t_j} T_j = F_T F_T^*.
\]
The operators $F_T$ and $F_T^*$ are obtained by left and right
stopping the process~$F$ at the quantum stopping time~$T$, so it is
not surprising that~$F_T F_T^*$ corresponds to double
stopping~$F$. The following results shows that this intuition holds
true in general.
\end{remark}

\begin{lemma}\label{lem:Fint}
The identity-adapted martingale $F$ is such that
\[
F_t = E_0 + \int_t^\infty I_\mul \otimes F_s \std \Lambda_s = %
I - \int_0^t I_\mul \otimes F_s \std \Lambda_s %
\qquad \text{for all } t \in [ 0, \infty ].
\]
\end{lemma}
\begin{proof}
Taking $M = E_0$ in Lemma~\ref{lem:Mint} gives the first
identity. From this it follows that
\[
I = F_0 = E_0 + \int_0^\infty I_\mul \otimes F_s \std \Lambda_s = %
E_0 + \int_0^t I_\mul \otimes F_s \std \Lambda_s + %
\int_t^\infty I_\mul \otimes F_s \std \Lambda_s,
\]
and rearranging this gives the second.
\end{proof}

\begin{theorem}\label{thm:FSint}
Given an arbitrary quantum stopping time $S$, there exist a
contraction $F_S \in \bop{\fock}{}$ such that
\[
F_S = S\bigl( \{ 0 \} \bigr) + \stlim_\pi %
\sum_{j = 1}^\infty S\bigl( ( \pi_{j - 1}, \pi_j ] \bigr) F_{\pi_j} %
+ S\bigl( \{ \infty \} \bigr) E_0,
\]
where the limit is taken over partitions of $\R_+$, ordered by
refinement. This operator $F_S$ is such that
\[
\langle F_S x, F_S y \rangle = \int_{[ 0, \infty ]} %
\langle x, F_s y \rangle \, \expn_\Vac\bigl[ S( \rd s ) \bigr] %
\qquad \text{for all } x, y \in \fock
\]
and
\[
F_S = E_0 + \int_0^\infty %
I_\mul \otimes S\bigl( [ 0, s ] \bigr) F_s \std \Lambda_s = %
I - \int_0^\infty %
I_\mul \otimes S\bigl( ( s, \infty ] \bigr) F_s \std \Lambda_s.
\]
\end{theorem}
\begin{proof}
Let
$F_S^\pi := S\bigl( \{ 0 \} \bigr) + %
\sum_{j = 1}^\infty S\bigl( ( \pi_{j - 1}, \pi_j ] \bigr) F_{\pi_j} %
+ S\bigl( \{ \infty \} \bigr) E_0$
and note that
\[
\| ( X \otimes I_{[t} ) F_t x \| = \| X \evec{0} \| \, \| F_t x \|
\]
for all $t \in ( 0, \infty )$, $X \in \bop{\fock_{t)}}{}$ and
$x \in \fock$, so $F_S^\pi$ converges in the strong operator topology
and~$\| F_S^\pi \| \le 1$. Let $\pi'$ be a refinement of $\pi$, let
$k_j$ and $l_j$ have the same meaning as they do in the proof of
Theorem~\ref{thm:expS}. Given~$f \in \elltwo$,
\begin{align*}
\| ( F_S^\pi - F_S^{\pi'} ) & \evec{f} \|^2 \\[1ex]
 & = \sum_{j = 0}^\infty \sum_{l = 1}^{l_j} %
\| S\bigl( ( \pi'_{k_j + l - 1}, \pi'_{k_j + l} ] \bigr) %
( F_{\pi'_{k_j + l_j}} - F_{\pi'_{k_j + l}} ) \evec{f} \|^2 \\[1ex]
 & = \sum_{j = 0}^\infty \sum_{l = 1}^{l_j} %
\| S\bigl( ( \pi'_{k_j + l - 1}, \pi'_{k_j + l} ] \bigr) %
\evec{0} \|^2 \, %
\| ( F_{\pi'_{k_j + l_j}} - F_{\pi'_{k_j + l}} ) \evec{f} \|^2.
\end{align*}
If $s$, $t \in \R_+$ are such that $s < t$ then
\begin{equation}\label{eqn:Fcty}
\| ( F_t - F_s ) \evec{f} \| = %
\| \evec{f|_{[ s, t )}} - \evec{0|_{[ s, t )}} \| \, %
\| \evec{f|_{[ t, \infty )}} \|
\end{equation}
and therefore
\[
\| ( F_S^\pi - F_S^{\pi'} ) \evec{f} \|^2 \le \| \evec{f} \|^2 \, 
\sup\{ \exp\bigl( \| 1_{[ \pi_j, \pi_{j + 1} )} f \|^2 \bigr) - 1 : %
j \ge 0 \},
\]
which tends to zero as $\pi$ is refined. Hence
$F_S \evec{f} := \lim_\pi F_S^\pi \evec{f}$ exists for all
$f \in \elltwo$; as exponential vectors are total in $\fock$ and each
$F_S^\pi$ is a contraction, the operator $F_S$ exists as claimed.

By (\ref{eqn:Fcty}), the map $s \mapsto F_s x$ is continuous for all
$x \in \fock$ and therefore
\begin{align*}
\langle F_S x, F_S y \rangle & = %
\lim_\pi \langle F_S^\pi x, F_S^\pi y \rangle \\[1ex]
 & = \langle x, S\bigl( \{ 0 \} \bigr) y \rangle + %
\lim_\pi \sum_{j = 1}^\infty \langle \evec{0}, %
S\bigl( ( \pi_{j - 1}, \pi_j ] \bigr) \evec{0} \rangle \, 
\langle x, F_{\pi_j} y \rangle \\
 & \hspace{18em} + %
\langle E_0 x, S\bigl( \{ \infty \} \bigr) E_0 y \rangle \\[1ex]
 & = \int_{[ 0, \infty ]} \langle x, F_s y \rangle \, %
\expn_\Vac\bigl[ S( \rd s ) \bigr]
\end{align*}
for all $x$, $y \in \fock$.

For the final claim, without loss of generality
suppose~$S\bigl( \{ 0 \} \big) = 0$, and let $f$, $g \in \elltwo$,
with $g$ having support in $[ 0, \pi_n ]$. Then 
$F_t \evec{g} = E_0 \evec{g}$ whenever~$t \ge \pi_n$, so it follows
that
\begin{align*}
\langle \evec{f}, ( F_S^\pi - E_0 ) \evec{g} \rangle & = %
\sum_{j = 1}^\infty \langle \evec{f}, %
S\bigl( ( \pi_{j - 1}, \pi_j ] \bigr) %
( F_{\pi_j} - E_0 ) \evec{g} \rangle \\[1ex]
 & = \sum_{j = 1}^{n - 1} \sum_{k = j}^{n - 1} \langle \evec{f}, %
S\bigl( ( \pi_{j - 1}, \pi_j ] \bigr) %
( F_{\pi_k} - F_{\pi_{k + 1}} ) \evec{g} \rangle \\[1ex]
 & = \sum_{k = 1}^{n - 1} \langle \evec{f}, %
\int_{\pi_k}^{\pi_{k + 1}} I_\mul \otimes %
S\bigl( [ 0, \pi_k ] \bigr) F_s \std \Lambda_s \, \evec{g} \rangle,
\end{align*}
where the final equality follows from by Lemma~\ref{lem:Fint} and
Lemma~\ref{lem:qsi}. The first identity is now seen to hold, since
\[
R_\pi' := \int_0^\infty \sum_{k = 0}^\infty %
1_{( \pi_k, \pi_{k + 1} ]}( s ) I_\mul \otimes %
S\bigl( ( \pi_k, s ] \bigr) F_s \std \Lambda_s \, \evec{g} \to 0
\]
as $\pi$ is refined, by Lemma~\ref{lem:qsiestimate}, working as in the
proof of Theorem~\ref{thm:qstint}. The second identity is now a
consequence of Lemma~\ref{lem:Fint}.
\end{proof}

\begin{theorem}
Let $S$ be a quantum stopping time. Then
\[
F_S F_S^* = E_0 + \int_0^\infty I_\mul \otimes %
S\bigl( [ 0, s ] \bigr) F_s S\bigl( [ 0, s ] \bigr) \std \Lambda_s %
= F_{\wh{S}}.
\]
\end{theorem}
\begin{proof}
The second identity is immediate from (\ref{eqn:idstopop}) and the
fact that $F$ is an identity-adapted martingale closed by
$E_0 = F_{\wc{S}}$. For the first, note first that~$F_S E_0 = E_0$
and, working as in the proof of Theorem~\ref{thm:FSint},
\[
\int_0^t I_\mul \otimes S\bigl( [ 0, s ] \bigr) F_s \std \Lambda_s = %
S\bigl( [ 0, t ] \bigr) ( F_S - F_t ) %
\qquad \text{for all } t \in [ 0, \infty ].
\]
In particular, $S\bigl( [ 0, t ] \bigr) ( F_S - F_t ) F_t = 0$ and
therefore
\begin{align*}
F_S F_S^* - E_0 & = ( F_S - E_0 ) ( F_S - E_0 )^* \\[1ex]
 & = \int_0^\infty I_\mul \otimes S\bigl( [ 0, s ] \bigr) F_s %
\std \Lambda_s \int_0^\infty I_\mul \otimes F_s %
S\bigl( [ 0, s ] \bigr) \std \Lambda_s \\[1ex]
 & = \int_0^\infty I_\mul \otimes \Bigl( %
S\bigl( [ 0, s ] \bigr) ( F_S - F_s ) F_s S\bigl( [ 0, s ] \bigr) \\
 & \hspace{3em} + %
S\bigl( [ 0, s ] \bigr) F_s ( F_S^* - F_s ) S\bigl( [ 0, s ] \bigr) %
 + S\bigl( [ 0, s ] \bigr) F_s S\bigl( [ 0, s ] \bigr) \Bigr) %
\std \Lambda_s \\
 & = \int_0^\infty I_\mul \otimes %
S\bigl( [ 0, s ] \bigr) F_s S\bigl( [ 0, s ] \bigr) \std \Lambda_s.
\qedhere
\end{align*}
\end{proof}

\appendix

\section{Quantum stochastic calculus}\label{sec:qsc}

\begin{definition}\label{def:vacrsm}
Let $\mul_1$ and $\mul_2$ be non-zero subspaces of $\mul$. A
\emph{$\mul_1-\mul_2$ process} $X$ is a family of bounded operators
$( X_t )_{t \in \R_+} \subseteq %
\bop{\mul_1 \otimes \fock}{\mul_2 \otimes \fock}$
such that $t \mapsto X_t \evec{f}$ is strongly measurable for all
$f \in \elltwo$.

A $\mul_1-\mul_2$ process $X$ is \emph{vacuum adapted} if
$( I_{\mul_2} \otimes E_t ) X_t ( I_{\mul_1} \otimes E_t ) = X_t$ for
all $t \in \R_+$.

Let $N$, $P$, $Q$ and $R$ be vacuum adapted $\mul-\mul$, $\mul-\C$,
$\C-\mul$ and $\C-\C$ processes, respectively, such that
\[
\| N \|_\infty := \esssup\{ \| N_t \| : t \in \R_+ \} < \infty %
\qquad \text{and} \qquad %
\int_0^\infty %
\bigl( \| P_s \|^2 + \| Q_s \|^2 + \| R_s \| \bigr) \std s < \infty.
\]
The \emph{gauge integral}
$\int N \std \Lambda = ( \int_0^t N_s \std \Lambda_s )$,
\emph{annihilation integral}
$\int P \std A = ( \int_0^t P_s \std A_s )$,
\emph{creation integral}
$\int Q \std A^\dagger = ( \int_0^t Q_s \std A^\dagger_s )$
and \emph{time integral}
$\int R \std t = ( \int_0^t R_s \std s )$
are the unique vacuum-adapted $\C-\C$~processes such that, for all
$f$, $g \in \elltwo$ and $t \in [ 0, \infty ]$,
\begin{align}\label{eqn:firstintid}
\langle \evec{f}, %
\int_0^t N_s \std \Lambda_s \, \evec{g} \rangle & = %
\int_0^t \langle f( s ) \otimes \evec{f}, N_s \bigl( %
g( s ) \otimes \evec{g} \bigr) \rangle \std s \\[1ex]
\langle \evec{f}, \int_0^t P_s \std A_s \, \evec{g} \rangle & = %
\int_0^t \langle \evec{f}, %
P_s \bigl( g( s ) \otimes \evec{g} \bigr) \rangle \std s \\[1ex]
\langle \evec{f}, %
\int_0^t Q_s \std A^\dagger_s \, \evec{g} \rangle & = %
\int_0^t \langle f( s ) \otimes \evec{f}, %
Q_s \evec{g} \rangle \std s \\[1ex]
\label{eqn:lastintid} \text{and} \qquad %
\langle \evec{f}, \int_0^t R_s \std s \, \evec{g} \rangle & = %
\int_0^t \langle \evec{f}, R_s \evec{g} \rangle \std s.
\end{align}

A vacuum-adapted $\C-\C$ process $X$ is a
\emph{regular $\Omega$-semimartingale} if there exist vacuum-adapted
processes $N$, $P$, $Q$, and $R$ as above and such that
\[
X_t = \int_0^t N_s \std \Lambda_s + \int_0^t P_s \std A_s + %
\int_0^t Q_s \std A^\dagger_s + \int_0^t R_s \std s %
\qquad \text{for all } t \in [ 0, \infty ];
\]
we write
$X = \int N \std \Lambda + P \std A + Q \std A^\dagger + R \std t$ to
denote this.

The \emph{martingale part} of $X$ is the vacuum-adapted
martingale $M = \int N \std \Lambda + P \std A + Q \std A^\dagger$
which is closed by $M_\infty$, where
\[
M_t = \int_0^t N_s \std \Lambda_s + \int_0^t P_s \std A_s + %
\int_0^t Q_s \std A^\dagger_s \qquad \text{for all } %
t \in [ 0, \infty ],
\]
and the \emph{FV part} of $X$ is the vacuum-adapted FV process
$Y = \int R \std t$ which is closed by
\[
Y_\infty = \int_0^\infty R_s \std s.
\]
\end{definition}

\begin{theorem}\label{thm:qipf}
Given regular $\Omega$-semimartingales
$X = \int N \std \Lambda + P \std A + Q \std A^\dagger + R \std t$
and~$X' = %
\int N' \std \Lambda + P' \std A + Q' \std A^\dagger + R' \std t$,
the process
$X X' = ( X_t X'_t )_{t \in [ 0, \infty ]}$ is a regular
$\Omega$-semimartingale, with
\[
X X' = \int N N' \std \Lambda + ( X P' + P N' ) \std A + %
( Q X' + N Q' ) \std A^\dagger + ( X R'+ R X' + P Q' ) \std t.
\]
\end{theorem}

\begin{definition}
A $\mul_1-\mul_2$ process $X$ is \emph{identity adapted} if and only
if
\[
\langle \evec{f}, X_t \evec{g} \rangle = %
\langle \evec{1_{[ 0, t )} f}, X_t \evec{1_{[ 0, t )} g} \rangle \, %
\langle \evec{1_{[ t, \infty )} f}, \evec{1_{[ t, \infty )} g} \rangle
\]
for all $f$, $g \in \elltwo$ and $t \in \R_+$; equivalently,
$X_t = X_{t)} \otimes I_{[t}$, where $X_{t)} \in \bop{\fock_{t)}}{}$,
for all~$t \in \R_+$.

Given a quadruple of processes $N$, $P$, $Q$ and $R$ as in
Definition~\ref{def:vacrsm}, but identity adapted rather than vacuum
adapted, there exist identity-adapted gauge, annihilation, creation
and time integrals which satisfy the same inner-product
identities. However, integration may not preserve boundedness in this
case, other than for the time integral: the integrals will exist as
linear operators with domains containing $\evecs$ such that the
identities (\ref{eqn:firstintid}--\ref{eqn:lastintid}) hold. If
\[
M_t = \int_0^t N_s \std \Lambda_s + \int_0^t P_s \std A_s + %
\int_0^t Q_s \std A^\dagger_s
\]
is a bounded operator for all $t \in [ 0, \infty ]$ then the
identity-adapted $\C-\C$ process $X = M + Y$,
where~$Y = \int H \std t$, is a \emph{regular quantum semimartingale}.
\end{definition}

A weak form of It\^o product formula holds for quantum semimartingales
which are not necessarily regular. We only need the version for gauge
integrals, which is as follows.

\begin{theorem}\label{thm:weakqipf}
Let $N$ and $N'$ be identity-adapted $\mul-\mul$ processes such that
$\| N \|_\infty < \infty$ and~$\| N' \|_\infty < \infty$. If
$X_t = \int_0^t N_s \std \Lambda_s$ and 
$X'_t = \int_0^t N'_s \std \Lambda_s$ then
\[
\langle X_t x, X_t' x' \rangle = \!\int_0^t \bigl( %
\langle ( I_\mul \otimes X_s ) \nabla_s x, %
N'_s \nabla_s x' \rangle + %
\langle N_s \nabla_s x, %
( I_\mul \otimes X'_s ) \nabla_s x' \rangle + %
\langle N_s \nabla_s x, N'_s \nabla_s x' \rangle \bigr) %
\std s
\]
for all $t \in [ 0, \infty ]$ and $x$, $x' \in \fock$, where
\[
\nabla : \evecs \to L^2( \R_+; \mul \otimes \fock ); \ %
\big( \nabla \evec{f} \bigr)( t ) = \nabla_t \evec{f} := %
f( t ) \evec{f}
\]
is the linear gradient operator.
\end{theorem}
\begin{proof}
See \cite[Theorem~3.15]{Lin05}.
\end{proof}

\begin{remark}\label{rem:independence}
The decomposition $X = M + Y$ is unique for both regular
$\Omega$-semimartingales and regular quantum semimartingales, by
Lemma~\ref{lem:decomp}. In fact, more may be shown: the quantum
stochastic integrators are independent, in the sense that if
\[
\int_0^\infty N_s \std \Lambda_s + \int_0^\infty P_s \std A_s + %
\int_0^\infty Q_s \std A^\dagger_s + \int_0^\infty R_s \std s = 0
\]
then $N_s = P_s = Q_s = R_s = 0$ for almost all $s \in \R_+$
\cite{Lin91}.
\end{remark}

\begin{lemma}\label{lem:qsi}
Let $t \in \R_+$ and $Z$, $W \in \bop{\fock_{t)}}{} \otimes I_{[t}$.
If $( N_s )_{s \in \R_+}$ is a vacuum or identity-adapted
$\mul-\mul$~process such that $\| N \|_\infty < \infty$ then so is
$\bigl( 1_{[ t, \infty )}( s ) ( I_\mul \otimes Z ) N_s %
( I_\mul \otimes W ) \bigr)_{s \in \R_+}$,
with the same type of adaptedness, and
\[
Z \int_t^\infty N_s \std \Lambda_s \, W = %
\int_t^\infty ( I_\mul \otimes Z ) N_s ( I_\mul \otimes W ) %
\std \Lambda_s.
\]
\end{lemma}
\begin{proof}
The first claim is immediate. For the second, suppose first that
$Z = W( f )$, the Weyl operator corresponding to~$f \in \elltwo$, as
in Example~\ref{eg:weyl}. Then $Z^* = W( -f )$ and, if $f$ has support
in~$[ 0, t ]$,
\begin{align*}
\langle \evec{g}, %
Z \int_t^\infty N_s \std & \Lambda_s \, \evec{h} \rangle \\[1ex]
 & = \exp\bigl( -\hlf \| f \|^2 + \langle f, g \rangle \bigr) %
\langle \evec{g - f}, \int_t^\infty N_s \std \Lambda_s \, %
\evec{h} \rangle \\[1ex]
 & = \exp\bigl( -\hlf \| f \|^2 + \langle f, g \rangle \bigr) %
\int_t^\infty \langle ( g - f )( s ) \otimes \evec{g - f}, %
N_s \bigl( h( s ) \otimes \evec{h} \bigr) \rangle \std s \\[1ex]
 & = \int_t^\infty \langle g( s ) \otimes Z^* \evec{g}, %
N_s \bigl( h( s ) \otimes \evec{h} \bigr) \rangle \std s \\[1ex]
 & = \langle \evec{g}, %
\int_t^\infty ( I_\mul \otimes Z ) N_s \std \Lambda_s \, %
\evec{h} \rangle
\end{align*}
for all $g$, $h \in \elltwo$. Since such Weyl operators are dense in
$\bop{\fock_{t)}}{} \otimes I_{[t}$ for the strong operator topology,
the result holds for general $Z$ and $W = I$; the full version may
be obtained by taking adjoints.
\end{proof}

\begin{lemma}\label{lem:qsiestimate}
If $N$ is a vacuum or identity-adapted $\mul-\mul$~process such that
$\| N \|_\infty < \infty$ then
\[
\| \int_0^\infty N_s \std \Lambda_s \, \evec{f} \|^2 \le C_f^2 %
\int_0^\infty \| N_s \bigl( f( s ) \otimes \evec{f} \bigr) \|^2 %
\std s \qquad \text{for all } f \in \elltwo,
\]
where $C_f := 1$ if $N$ is vacuum adapted and
$C_f := \| f \| + \sqrt{1 + \| f \|^2}$ if $N$ is identity adapted.
\end{lemma}
\begin{proof}
See \cite[Proof of Theorem~18]{Blt04} and \cite[Theorem~3.13]{Lin05},
respectively.
\end{proof}

\begin{theorem}\label{thm:switch}
If $X = \int N \std \Lambda + P \std A + Q \std A^\dagger + R \std t$
is a regular $\Omega$-semimartingale
then~$\wh{\pi}( X )$ is a regular quantum semimartingale such that
\[
\wh{\pi}( X )_t = %
\int_0^t \wh{\pi}( N - I_\mul \otimes X )_s \std \Lambda_s + %
\int_0^t \wh{\pi}( P )_s \std A_s + %
\int_0^t \wh{\pi}( Q )_s \std A^\dagger_s + %
\int_0^t \wh{\pi}( R )_s \std s
\]
for all $t \in [ 0, \infty ]$.
\end{theorem}
\begin{proof}
See \cite[Corollaries~31 and~40]{Blt04}; the extension to a
non-separable multiplicity space $\mul$ is straightforward.
\end{proof}

\begin{notation}\label{not:poisson}
Let $\mul = \C$, let $\nu = ( \nu_t )_{t \in \R_+}$ be a standard
Poisson process on the probability space $\probsp$ and let
$U_P : L^2\probsp \to \fock$ be the isometric isomorphism such
that~the closure of
$N_t := \Lambda_t + A_t + A^\dagger_t + t I$ equals $ U_P \nu_t U_P^*$
for all $t \in \R_+$, where $\wh{\nu_t}$ acts by multiplication
by~$\nu_t$, and $U_P 1 = \evec{0}$ \cite[Theorems~6.1--2]{HuP84}.
Recall that
\[
\sexp{f} := U_P^* \evec{f} = %
1 + \int_0^\infty f( t ) \sexp{1_{[ 0, t )} f} \std \chi_t %
\qquad \text{for all } f \in L^2( \R_+ ),
\]
where $\chi$ is the normal martingale such that $\chi_t = \nu_t - t$
for all $t \in \R_+$ \cite[Section~II.1]{Att98}.
\end{notation}

\begin{lemma}\label{lem:poisint}
In the setting of Notation~\ref{not:poisson}, let $\phi$ be a
bounded process on $\probsp$ adapted to the Poisson filtration and let
$F_t = U_P \wh{\phi}_t U_P^* \in \bop{\fock}{}$ for all $t \in \R_+$,
where $\wh{\phi}_t$ acts as multiplication by~$\phi_t$. Then
\begin{equation}\label{eqn:poisson}
U_P \int_0^\infty \phi_t \std \chi_t U_P^* = %
\int_0^\infty F_t \std ( N_t - t ) \quad \text{on } \evecs.
\end{equation}
\end{lemma}
\begin{proof}
For all $f \in L^2( \R )$ and $t \in \R_+$, let
\[
E_t \sexp{f} := \sexp{1_{[ 0, t )} f} \qquad \text{and} \qquad %
D_t \sexp{f} := f( t ) E_t \sexp{f}.
\]
Recall that the quadratic variation $[ \chi ]_t = \chi_t + t$ for all
$t \in \R_+$.

Thus, if~$f$, $g \in L^2( \R_+ )$ then
\begin{align*}
&\expn\Bigl[ \overline{\sexp{f}} %
\int_0^\infty \phi_t \std \chi_t \, \sexp{g} \Bigr] \\[1ex]
 & = \expn\Bigl[ %
\Bigl( 1 + \int_0^\infty D_t \overline{\sexp{f}} \std \chi_t \Bigr) %
\int_0^\infty \phi_t \std \chi_t \, \sexp{g} \Bigr] \\[1ex]
 & = \expn\Bigl[ \Bigl( \int_0^\infty \phi_t \std \chi_t + %
\int_0^\infty D_t \overline{\sexp{f}} \int_0^t \phi_s \std \chi_s %
\std \chi_t \\
 & \hspace{6em} + \int_0^\infty %
\int_0^t D_s \overline{\sexp{f}} \std \chi_s \, \phi_t \std \chi_t + %
\int_0^\infty D_t \overline{\sexp{f}} \phi_t \std ( \chi_t + t ) %
\Bigr) \Bigl( 1 + %
\int_0^\infty D_t \sexp{g} \std \chi_t \Bigr) \Bigr] \\[1ex]
 & = \int_0^\infty \expn\bigl[ %
D_t \overline{\sexp{f}} \phi_t + \phi_t D_t \sexp{g} + %
D_t \overline{\sexp{f}} \int_0^t \phi_s \std \chi_s \, %
D_t \sexp{g} + \int_0^t D_s \overline{\sexp{f}} \std \chi_s \, %
\phi_t D_t \sexp{g} \\
 & \hspace{19em} + D_t \overline{\sexp{f}} \phi_t D_t \sexp{g} + %
D_t \overline{\sexp{f}} \phi_t %
\int_0^t D_s \sexp{g} \std \chi_s \bigr] \std s \\[1ex]
 & = \int_0^\infty \expn\bigl[ %
D_t \overline{\sexp{f}} \phi_t E_t \sexp{g} + %
E_t \overline{\sexp{f}} \phi_t D_t \sexp{g} + %
+ D_t \overline{\sexp{f}} \phi_t D_t \sexp{g} + %
D_t \overline{\sexp{f}} \smash[t]{\int_0^t} %
\phi_s \std \chi_s D_t \sexp{g} \bigr] \std t \\[1ex]
 & = \int_0^\infty \expn\bigl[ %
E_t \overline{\sexp{f}} \Bigl( \overline{f( t )} \phi_t + %
\phi_t g( t ) + \overline{f( t )} \Bigl( \phi_t + %
\int_0^t \phi_s \std \chi_s \Bigr) g( t ) \Bigr) %
E_t \sexp{g} \bigr] \std t.
\end{align*}
Replacing $f$, $g$ and $\phi$ by $1_{[ 0, t ]} f$, $1_{[ 0, t ]} g$
and $1_{[ 0, t ]} \phi$, respectively, it follows that
\begin{align*}
\alpha_t & := \expn\Bigl[ \overline{\sexp{1_{[ 0, t ]} f}} %
\int_0^t \phi_s \std \chi_s \, \sexp{1_{[ 0, t ]} g} \Bigr] \\[1ex]
 & \hphantom{:}= \int_0^t \bigl( \overline{f( s )} + g( s ) + %
\overline{f( s )} g( s ) \bigr) %
\expn[ E_s \overline{\sexp{f}} \phi_s E_s \sexp{g} ] + %
\overline{f( s )} g( s ) \alpha_s \std s,
\end{align*}
so
\[
\frac{\rd}{\rd t}\Bigl( \alpha_t %
\exp( \int_t^\infty \overline{f( s )} g( s ) \std s ) \Bigr) = %
\Bigl( \overline{f( t )} + g( t ) + \overline{f( s )} g( s ) \Bigr) %
\expn\Bigl[ \overline{\sexp{f}} \phi_t \sexp{f} \Bigr]
\]
and
\[
\expn\Bigl[ \overline{\sexp{f}} %
\int_0^\infty \phi_t \std \chi_t \, \sexp{g} \Bigr] = %
\int_0^\infty \Bigl( \overline{f( t )} + g( t ) + %
\overline{f( t )} g( t ) \Bigr) %
\expn\Bigl[ \overline{\sexp{f}} \phi_t \sexp{g} \Bigr] \std t.
\qedhere
\]
\end{proof}


\section*{References}

\begin{thebibliography}{99}

\bibitem{AcS89}
\textsc{L.~Accardi} \& \textsc{K.~Sinha},
Quantum stop times,
in
\textit{Quantum Probability and Applications~IV}
(L.~Accardi \& W.~von~Waldenfels, eds.), 68--72,
Lecture Notes in Math.~1396, Springer, Berlin, 1989.

\bibitem{App88}
\textsc{D.~Applebaum},
Stopping unitary processes in Fock space,
\textit{Publ.\ RIMS Kyoto Univ.}~24 (1988), 697--705.

\bibitem{Att98}
\textsc{S.~Attal},
Classical and quantum stochastic calculus,
in
\textit{Quantum Probability Communications~X}
(R.L.~Hudson \& J.M.~Lindsay, eds.), 1--52,
World Scientific, Singapore, 1998.

\bibitem{AtC04}
\textsc{S.~Attal} \& \textsc{A.~Coquio},
Quantum stopping times and quasi-left continuity,
\textit{Ann.\ Inst.\ H.\ Poincar\'e Probab.\ Statist.}~40 (2004),
no.4, 497--512.

\bibitem{AtL04}
\textsc{S.~Attal} \& \textsc{J.M.~Lindsay},
Quantum stochastic calculus with maximal operator domains,
\textit{Ann.\ Probab.}~32 (2004), no.1A, 488--529.

\bibitem{AtP95}
\textsc{S.~Attal} \& \textsc{K.R.~Parthasarathy},
Strong Markov processes and the Dirichlet problem in von Neumann
algebras,
in
\textit{Stochastic analysis and applications (Powys, 1995)}
(I.M.~Davies, A.~Truman \& K.D.~Elworthy, eds.), 53--75, World
Scientific, Singapore, 1996.

\bibitem{AtP96}
\textsc{S.~Attal} \& \textsc{K.R.~Parthasarathy},
Strong Markov processes and the Dirichlet problem in $C^*$-algebras,
preprint no.357, Institut Fourier, Universit\'e Grenoble Alpes, 1996.

\bibitem{AtS98}
\textsc{S.~Attal} \& \textsc{K.B.~Sinha},
Stopping semimartingales on Fock space,
in
\textit{Quantum Probability Communications~X}
(R.L.~Hudson \& J.M.~Lindsay, eds.), 171--185,
World Scientific, Singapore, 1998.

\bibitem{BaL86}
\textsc{C.~Barnett} \& \textsc{T.~Lyons},
Stopping noncommutative processes,
\textit{Math.\ Proc.\ Cambridge Philos.\ Soc.}~99 (1986), no.1,
151--161.

\bibitem{BaT87}
\textsc{C.~Barnett} \& \textsc{B.~Thakrar}
Time projections in a von~Neumann algebra,
\textit{J.\ Operator Theory}~18 (1987), no.1, 19--31.

\bibitem{BaT90}
\textsc{C.~Barnett} \& \textsc{B.~Thakrar},
A noncommutative random stopping theorem,
\textit{J.\ Funct.\ Anal.}~88 (1990), no.2, 342--350.

\bibitem{BaW90}
\textsc{C.~Barnett} \& \textsc{I.F.~Wilde},
Random times and time projections,
\textit{Proc.\ Amer.\ Math.\ Soc.}~110 (1990), no.2, 425--440.

\bibitem{BaW93}
\textsc{C.~Barnett} \& \textsc{I.~Wilde},
Random times, predictable processes and stochastic integration in
finite von~Neumann algebras,
\textit{Proc.\ London Math.\ Soc.}~(\textit{3})~67 (1993), no.2,
355--383.

\bibitem{Blt01}
\textsc{A.C.R.~Belton},
Quantum $\Omega$-semimartingales and stochastic evolutions,
\textit{J.\ Funct.\ Anal.}~187 (2001), no.1, 94--109.

\bibitem{Blt04}
\textsc{A.C.R.~Belton},
An isomorphism of quantum semimartingale algebras,
\textit{Q.\ J.\ Math.}~55 (2004), no.2, 135--165.

\bibitem{BeS14}
\textsc{A.C.R.~Belton} \& \textsc{K.B.~Sinha},
Stopping the CCR flow and its isometric cocycles,
\textit{Q.\ J.\ Math.}~65 (2014), no.4, 1145--1164.

\bibitem{BeS16}
\textsc{A.C.R.~Belton} \& \textsc{K.B.~Sinha},
The cocycle identity holds under stopping,
\textit{Q.\ J.\ Math.}~68 (2017), no.3, 817--830.

\bibitem{BhP95}
\textsc{B.V.R.~Bhat} \& \textsc{K.R.~Parthasarathy},
Markov dilations of nonconservative dynamical semigroups and a quantum
boundary theory,
\textit{Ann.\ Inst.\ H.\ Poincar\'e Probab.\ Statist.}~31 (1995),
no.4, 601--651.

\bibitem{Coq06}
\textsc{A.~Coquio},
The optional stopping theorem for quantum martingales,
\textit{J.\ Funct.\ Anal.}~238 (2006), no.1, 149--180.

\bibitem{Hud79}
\textsc{R.L.~Hudson},
The strong Markov property for canonical Wiener processes,
\textit{J.\ Funct.\ Anal.}~34 (1979), no.2, 266--281.

\bibitem{Hud07}
\textsc{R.L.~Hudson},
Stop times in Fock space quantum probability,
\textit{Stochastics}~79 (2007), no.3--4, 383--391.

\bibitem{Hud10}
\textsc{R.L.~Hudson},
Stopping Weyl processes,
in
\textit{Recent development in stochastic dynamics and stochastic
analysis}
(J.~Duan, S.L.~Luo and C.~Wang, eds.), 185--193,
World Scientific, Singapore, 2010.

\bibitem{HuP84}
\textsc{R.L.~Hudson} \& \textsc{K.R.~Parthasarathy},
Quantum Ito's formula and stochastic evolutions,
\textit{Comm.\ Math.\ Phys.}~93 (1984), no.3, 301--323.

\bibitem{Lin91}
\textsc{J.M.~Lindsay},
Independence for quantum stochastic integrators,
in
\textit{Quantum Probability \& Related Topics~VI}
(L.~Accardi, ed.), 325--332,
World Scientific, Singapore, 1991.

\bibitem{Lin05}
\textsc{J.M.~Lindsay},
Quantum stochastic analysis --- an introduction,
in
\textit{Quantum Independent Increment Processes I}
(M.~Sch\"urmann \& U.~Franz, eds.), 181--271,
Lecture Notes in Math.~1865, Springer, Berlin, 2005.

\bibitem{Luc05}
\textsc{A.~{\L}uczak},
Structure of the time projection for stopping times in von Neumann
algebras,
\textit{Internat.\ J.\ Theoret.\ Phys.}~44 (2005), no.7, 909--917.

\bibitem{Mey86}
\textsc{P.-A.~Meyer},
\'El\'ements de probabilit\'es quantiques~IV. Probabilit\'es sur
l'espace de Fock,
in
\textit{S\'eminaire de Probabilit\'es~XX}
(J.~Az\'ema \& M.~Yor, eds.), 249--285,
Lecture Notes in Math.~1204, Springer, Berlin, 1986.

\bibitem{Mey87}
\textsc{P.-A.~Meyer},
\'El\'ements de probabilit\'es quantiques~VIII. Temps d'arr\^et sur
l'espace de Fock,
in
\textit{S\'eminaire de Probabilit\'es~XXI}
(J~Az\'ema, P.-A.~Meyer \& M.~Yor, eds.), 63--78,
Lecture Notes in Math.~1247, Springer, Berlin, 1987.

\bibitem{Par03}
\textsc{K.R.~Parthasarathy},
Quantum probability and strong quantum Markov processes,
in
\textit{Quantum Probability Communications~XII}
(S.~Attal \& J.M.~Lindsay, eds.), 59--138,
World Scientific, Singapore, 2003

\bibitem{PaS87}
\textsc{K.R.~Parthasarathy} \& \textsc{K.B.~Sinha},
Stop times in Fock space stochastic calculus,
\textit{Probab.\ Theory Related Fields} 75 (1987), no.3, 317--349.

\bibitem{Sav88}
\textsc{J.-L.~Sauvageot},
First exit time: a theory of stopping times in quantum processes,
in
\textit{Quantum Probability and Applications~III}
(L.~Accardi \& W.~von~Waldenfels, eds.), 285--299,
Lecture Notes in Math.~1303, Springer, Berlin, 1988.

\bibitem{Sin03}
\textsc{K.B.~Sinha},
Quantum stop times,
in
\textit{Quantum Probability Communications~XII}
(S.~Attal \& J.M.~Lindsay, eds.), 195--207,
World Scientific, Singapore, 2003.

\end{thebibliography}
\end{document}